\documentclass{amsart}

\input xy

\xyoption{all}

\usepackage{amssymb,amsbsy,amsthm,amsmath,graphicx,epsfig}

\newtheorem*{thm*}{Theorem}
\newtheorem{theorem}{Theorem}[section]
\newtheorem{lem}[theorem]{Lemma}
\newtheorem{lemma}[theorem]{Lemma}
\newtheorem{prop}[theorem]{Proposition}
\newtheorem{proposition}[theorem]{Proposition}
\newtheorem{cor}[theorem]{Corollary}
\newtheorem{corollary}[theorem]{Corollary}

\newtheorem{dfn}[theorem]{Definition}
\newtheorem{definition}[theorem]{Definition}

\newtheorem{problem}[theorem]{Problem}
\newtheorem{ques}[theorem]{Question}
\theoremstyle{remark}

\newtheorem{remark}{Remark}[section]
\newtheorem{remarks}{Remarks}[section]
\newtheorem{example}[remark]{Example}

\renewcommand{\rm}[1]{\mathrm{#1}}
\newcommand{\cal}[1]{\mathcal{#1}}

\newcommand{\bbC}{\mathbb{C}}

\newcommand{\bbF}{\mathbb{F}}
\newcommand{\bbN}{\mathbb{N}}

\newcommand{\bbR}{\mathbb{R}}
\newcommand{\bbT}{\mathbb{T}}
\newcommand{\bbZ}{\mathbb{Z}}
\newcommand{\one}{\mathbf{1}}

\renewcommand{\d}{\mathrm{d}}

\newcommand{\id}{\mathrm{id}}

\newcommand{\A}{\mathcal{A}}

\newcommand{\N}{\mathcal{N}}
\newcommand{\M}{\mathcal{M}}

\newcommand{\Z}{\mathcal{Z}}

\newcommand{\frH}{\mathfrak{H}}

\renewcommand{\a}{\alpha}

\newcommand{\eps}{\varepsilon}
\newcommand{\g}{\gamma}

\newcommand{\s}{\sigma}

\newcommand{\lin}{\mathrm{lin}}

\renewcommand{\hat}[1]{\widehat{#1}}
\newcommand{\ol}[1]{\overline{#1}}

\newcommand{\into}{\hookrightarrow}
\newcommand{\fin}{\nolinebreak\hspace{\stretch{1}}$\lhd$}

\renewcommand{\Re}{\operatorname{Re}}

\begin{document}

\title[Von Neumann nonconventional averages]{Nonconventional ergodic averages and multiple recurrence
for von Neumann dynamical systems}
\author{Tim Austin \and Tanja Eisner \and Terence Tao}
\thanks{T.A. is supported by a fellowship from Microsoft Corporation. T.E. is supported by the European Social Fund
and by the Ministry of Science, Research and the Arts
Baden-W\"urttemberg. T.T. is supported by NSF grant DMS-0649473
and a grant from the Macarthur Foundation.}
\date{}

\begin{abstract}  The Furstenberg recurrence theorem (or equivalently, Szemer\'edi's theorem) can be
formulated in the language of von Neumann algebras as follows: given
an integer $k \geq 2$, an abelian finite von Neumann algebra
$(\M,\tau)$ with an automorphism $\alpha: \M \to \M$, and a
non-negative $a \in \M$ with $\tau(a)>0$, one has $\liminf_{N \to
\infty} \frac{1}{N} \sum_{n=1}^N \Re \tau(a \alpha^n (a) \ldots
\alpha^{(k-1)n} (a)) > 0$; a subsequent result of Host and Kra shows
that this limit exists.  In particular, $\Re \tau(a \alpha^n (a)
\ldots \alpha^{(k-1)n} (a)) > 0$ for all $n$ in a set of positive
density.

From the von Neumann algebra perspective, it is thus natural to
ask to what extent these results remain true when the abelian
hypothesis is dropped. All three claims hold for $k = 2$, and we
show in this paper that all three claims hold for all $k$ when the
von Neumann algebra is asymptotically abelian, and that the last
two claims hold for $k=3$ when the von Neumann algebra is ergodic.
However, we show that the first claim can fail for $k=3$ even with
ergodicity, the second claim can fail for $k \geq 4$ even when assuming
ergodicity, and the third claim can fail for $k=3$ without
ergodicity, or $k \geq 5$ and odd assuming ergodicity. The second
claim remains open for non-ergodic systems with $k=3$, and the
third claim remains open for ergodic systems with $k=4$.
\end{abstract}

\maketitle

\tableofcontents

\section{Introduction}

\subsection{Multiple recurrence}

Let $(X, {\mathcal X}, \mu)$ be a probability space, and let $T: X
\to X$ be a measure-preserving invertible transformation on $X$
(i.e. $T, T^{-1}$ are both measurable, and $\mu(T(A)) = \mu(A)$
for all measurable $A$). From the mean ergodic theorem we know
that for any $f \in L^\infty(X)$, the averages\footnote{The minus
sign here is not of particular significance (other than to
conform to some minor notational conventions) and can be ignored
in the sequel if desired.} $\frac{1}{N} \sum_{n=1}^N f \circ
T^{-n}$ converge in (say) $L^2(X)$ norm, which implies in
particular that the averages $\frac{1}{N} \sum_{n=1}^N \int_X f_1
(f_2 \circ T^{-n})\ d\mu$ converge for all $f_1,f_2 \in
L^\infty(X)$.  Furthermore, if $f_1=f_2=f$ is non-negative with
positive mean $\int_X f\ d\mu > 0$, then the Poincar\'e recurrence
theorem implies that this latter limit is strictly positive.  In
particular, this implies that the mean $\int_X f (f \circ T^{-n})\
d\mu$ is positive for all natural numbers $n$ in a set $E \subset
\bbN$ of positive (lower) density (which means that $\liminf_{N \to
\infty} \frac{1}{N} \# \{ 1 \leq n \leq N: n \in E \} > 0$).

Thanks to a long effort starting with Furstenberg's groundbreaking
new proof \cite{Fur77} of Szemer\'edi's theorem on arithmetic
progressions \cite{Sze75}, it is now known that all of these
single recurrence results extend to multiple recurrence:

\begin{theorem}[Abelian multiple recurrence]\label{comm}  Let $(X, {\mathcal X},\mu)$ be a probability space, let
$k \geq 2$ be an integer, and let $T: X \to X$ be a measure-preserving invertible transformation.
\begin{itemize}
\item (Convergence in norm)  For any $f_1,\ldots,f_{k-1} \in L^\infty(X)$, the averages
$$ \frac{1}{N} \sum_{n=1}^N (f_1 \circ T^{-n}) \ldots (f_{k-1} \circ T^{-(k-1)n})$$
converge in $L^2(X)$ norm as $N \to \infty$.
\item (Weak
convergence)  For any $f_0,f_1,\ldots,f_{k-1} \in L^\infty(X)$,
the averages
$$ \frac{1}{N} \sum_{n=1}^N \int_X f_0 (f_1 \circ T^{-n}) \ldots (f_{k-1} \circ T^{-(k-1)n})\ d\mu$$
converge as $N \to \infty$.
\item (Recurrence on average)  For any non-negative $f \in L^\infty(X)$ with $\int_X f\ d\mu > 0$, one has
\begin{equation}\label{limf}
\liminf_{N \to \infty} \frac{1}{N} \sum_{n=1}^N \int_X f (f \circ T^{-n}) \ldots (f \circ T^{-(k-1)n})\ d\mu > 0.
\end{equation}
\item (Recurrence on a dense set)  For any non-negative $f \in L^\infty(X)$ with $\int_X f\ d\mu > 0$, one has
\begin{equation}\label{smash}
 \int_X f (f \circ T^{-n}) \ldots (f \circ T^{-(k-1)n})\ d\mu > c > 0
\end{equation}
for some $c>0$ and all $n$ in a set of natural numbers of positive lower density.
\end{itemize}
\end{theorem}

We have called this result the ``abelian'' multiple recurrence theorem to emphasise the abelian nature of the algebra
$L^\infty(X)$.

\begin{remarks}  Clearly, convergence in norm implies weak convergence; also, as the averages \eqref{smash} are
bounded and non-negative, recurrence on average implies recurrence
on a dense set.  Using the weak convergence result, the limit
inferior in \eqref{limf} can be replaced with a limit, but we have
retained the limit inferior in order to keep the two claims
logically independent of each other.

As mentioned earlier, the $k=2$ cases of Theorem \ref{comm} follow
from classical ergodic theorems. Furstenberg~\cite{Fur77}
established recurrence on average (and hence recurrence on a dense
set) for all $k$, and observed that this result was equivalent (by
what is now known as the \emph{Furstenberg correspondence
principle}) to Szemer\'edi's famous theorem \cite{Sze75} on
arithmetic progressions, thus providing an important new proof of
that theorem.  Convergence in norm (and hence in mean) was
established for $k=3$ by Furstenberg \cite{Fur77}, for $k=4$ by
Conze and Lesigne \cite{conze-lesigne}, \cite{conze2},
\cite{conze3} (assuming total ergodicity) and by Host and
Kra~\cite{HosKra01} (in general), for $k=5$ in some cases by
Ziegler \cite{Zie05}, and for all $k$ by Host and
Kra~\cite{HosKra05} (and subsequently also by Ziegler
\cite{Zie07}).  See \cite{kra-icm} for a survey of these results,
and their relation to other topics such as dynamics of
nilsequences, and arithmetic progressions in number-theoretic sets
such as the primes. \fin
\end{remarks}

There is also a multidimensional generalisation of the above results to multiple commuting shifts:

\begin{theorem}[Abelian multidimensional multiple recurrence]\label{comm-multid}  Let $(X, {\mathcal X},\mu)$ be a
probability space, let $k \geq 2$ be an integer, and let
$T_0,\ldots,T_{k-1}: X \to X$ be a commuting system of
measure-preserving invertible transformations.
\begin{itemize}
\item (Convergence in norm)  For any $f_1,\ldots,f_{k-1} \in L^\infty(X)$, the averages
$$ \frac{1}{N} \sum_{n=1}^N T_0^n( (f_1 \circ T_1^{-n}) \ldots (f_{k-1} \circ T_{k-1}^{-n}) )$$
converge in $L^2(X)$ norm.
\item (Weak convergence)  For any
$f_0,f_1,\ldots,f_{k-1} \in L^\infty(X)$, the averages
$$ \frac{1}{N} \sum_{n=1}^N \int_X (f_0 \circ T_0^{-n}) (f_1 \circ T_1^{-n}) \ldots (f_{k-1} \circ T_{k-1}^{-n})\ d\mu$$
converge.
\item (Recurrence on average)  For any non-negative $f \in L^\infty(X)$ with $\int_X f\ d\mu > 0$, one has
\begin{equation}\label{limf-multi}
\liminf_{N \to \infty} \frac{1}{N} \sum_{n=1}^N \int_X (f \circ T_0^{-n}) (f \circ T_1^{-n}) \ldots
(f \circ T_{k-1}^{-n})\ d\mu > 0.
\end{equation}
\item (Recurrence on a dense set)  For any non-negative $f \in L^\infty(X)$ with $\int_X f\ d\mu > 0$, one has
\begin{equation}\label{smash-multi}
 \int_X (f \circ T_0^{-n}) (f \circ T_1^{-n}) \ldots (f \circ T_{k-1}^{-n})\ d\mu > c > 0
\end{equation}
for some $c>0$ and all $n$ in a set of natural numbers of positive lower density.
\end{itemize}
\end{theorem}

Of course, Theorem \ref{comm} is the special case of Theorem
\ref{comm-multid} when $T_i := T^i$.  It is often customary to
normalise $T_0$ to be the identity transformation (by replacing
each of the $T_i$ with $T_0^{-1} T_i$).

\begin{remarks}  The $k=2$ case is again classical.  Recurrence on average (and hence on a dense set) in this theorem
was established for all $k$ by Furstenberg and Katznelson
\cite{FurKat78}, which by the Furstenberg correspondence principle
implies a multidimensional version of Szemer\'edi's theorem, a
combinatorial proof of which in full generality has only been
obtained relatively recently in \cite{NagRodSch07} and \cite{Gow06}.
Convergence in norm (and weak convergence) was established for $k=3$
in \cite{conze-lesigne}, for some special cases of $k=4$ in
\cite{zhang-conv}, for all $k$ assuming total ergodicity in
\cite{FraKra05}, and for all $k$ unconditionally in
\cite{Tao08(nonconv)} (with subsequent proofs at \cite{Tow07},
\cite{Aus--nonconv}, \cite{Hos??}).  The results can fail if the
shifts $T_0,\ldots,T_{k-1}$ do not commute~\cite{blfail}.  Note that
non-commutativity of the shifts should \emph{not} be confused with
the non-commutativity of the underlying algebra, which is the focus
of this current paper. \fin
\end{remarks}

\subsection{Non-commutative analogues}

From the perspective of the theory of von Neumann algebras, the
space $L^\infty(X)$ appearing in the above theorems can be
interpreted as an abelian von Neumann algebra, with a finite trace
$\tau(f) := \int_X f\ d\mu$, and with an automorphism $T:
L^\infty(X)\to L^\infty(X)$ defined by $Tf := f \circ T^{-1}$.  It
is then natural to ask whether the above results can be extended to
non-abelian settings.  More precisely, we recall the following
definitions.

\begin{definition}[Non-commutative systems]  A \emph{finite von Neumann algebra} is a pair
$(\M, \tau)$, where $\M$ is a von Neumann algebra (i.e. an algebra
of bounded operators on a separable\footnote{In our applications,
the hypothesis of separability can be omitted, since one can always
pass to the separable subalgebra generated by a finite collection
$a_0,\ldots,a_{k-1}$ of elements and their shifts if desired.}
complex Hilbert space that contains the identity $1$, is closed
under adjoints, and is closed in the weak operator topology), and
$\tau: \M \to \bbC$ is a finite faithful trace (i.e. a linear map
with $\tau(a^*) = \overline{\tau(a)}$, $\tau(ab) = \tau(ba)$, and
$\tau(a^* a) \geq 0$ for all $a,b \in \M$, with $\tau(a^* a)=0$ if
and only if $a=0$ and $\tau(1)=1$). The operator norm of an element
$a \in \M$ is denoted $\|a\|$.  We say that an element $a \in \M$ is
\emph{non-negative} if one has $a=b^* b$ for some $b \in \M$.  An
element $a \in \M$ is \emph{central} if one has $ab=ba$ for all $b
\in \M$.  The set of all central elements is denoted $\Z(\M)$ and
referred to as the \emph{centre} of $\M$; the algebra $\M$ is
\emph{abelian} if $\Z(\M) = \M$.

A \emph{shift} $\a$ on a finite von Neumann algebra
$(\M,\tau)$ is trace-preserving $*$-automorphism, i.e. $\a$ is an
algebra isomorphism such that $\a(a^*) = \a(a)^*$ and $\tau(\a(a))
= \tau(a)$ for all $a \in \M$.  We say that the shift is
\emph{ergodic} if the invariant algebra $\{ a \in \M: \a(a) = a
\}$ consists only of the constants $\bbC 1$.  We refer to the
triple $(\M,\tau,\a)$ as a \emph{von Neumann $\bbZ$-system}, or
a \emph{von Neumann dynamical system}.  More generally, if
$\a_0,\ldots,\a_{k-1}$ are $k$ commuting shifts on $M$, we refer
to $(\M,\tau,\a_0,\ldots,\a_{k-1})$ as a \emph{von Neumann
$\bbZ^k$-system}.
\end{definition}

It is easy to verify that if $(X,{\mathcal X}, \mu)$ is a
(classical) probability space with a shift $T: X \to X$, then
$(L^\infty(X), \int_X \cdot\ d\mu, \circ T^{-1})$ is an (abelian
example of a) von Neumann dynamical system, and more generally if
$T_0,\ldots,T_{k-1}: X \to X$ are commuting shifts, then
$(L^\infty(X), \int_X \cdot\ d\mu, \circ T_0^{-1},\ldots, \circ
T_{k-1}^{-1})$ is an abelian example of a von Neumann
$\bbZ^k$-system. In fact, all abelian von Neumann dynamical systems
arise (up to isomorphism of the algebras) as such examples; see
Kadison and Ringrose~\cite[Chapter 5]{KadRin97}.

A finite von Neumann algebra $(\M,\tau)$ gives rise to an
inner product $\langle a, b\rangle := \tau(a^*b)$ on $\M$; the
properties of the trace ensure that this inner product is positive
definite. (We use the convention for a scalar product to be
conjugate linear in the first coordinate.) The Hilbert space
completion of $\M$ with respect to this inner product will be
referred to as $L^2(\tau)$. Note that $\a$ extends to a unitary
transformation on $L^2(\tau)$. In the abelian case when $\M =
L^\infty(X,{\mathcal X},\mu)$, then $L^2(\tau)$ can be canonically
identified with $L^2(X, {\mathcal X},\mu)$.

Inspired by Theorems \ref{comm}, \ref{comm-multid}, we now make
the following definitions:

\begin{definition}[Non-commutative recurrence and convergence]  Let $k \geq 2$ be an integer, $(\M,\tau,\a)$
be a von Neumann dynamical system, and $(\M,\tau,\a_0,\ldots,\a_{k-1})$ be a von Neumann $\bbZ^k$-system.
\begin{itemize}
\item We say that $(\M,\tau,\a)$ enjoys \emph{order $k$
convergence in norm} if for any $a_1,\ldots,a_{k-1} \in \M$, the
averages
$$ \frac{1}{N} \sum_{n=1}^N (\a^n (a_1)) (\a^{2n} (a_2)) \ldots (\a^{(k-1)n} (a_{k-1})) $$
converge in $L^2(\tau)$ as $N \to \infty$. \item We say that
$(\M,\tau,\a)$ enjoys \emph{order $k$ weak convergence} if for any
$a_0,a_1,\ldots,a_{k-1} \in \M$, the averages
$$ \frac{1}{N} \sum_{n=1}^N \tau( a_0 (\a^n (a_1)) (\a^{2n} (a_2)) \ldots (\a^{(k-1)n} (a_{k-1}))) $$
converge as $N \to \infty$. \item We say that $(\M,\tau,\a)$
enjoys \emph{order $k$ recurrence on average} if for any
non-negative $a \in \M$ with $\tau(a) > 0$ one has
\begin{equation}\label{reco}
 \liminf_{N \to \infty} \frac{1}{N} \sum_{n=1}^N \Re \tau(a (\a^n (a)) (\a^{2n} (a)) \ldots (\a^{(k-1)n} (a))) > 0.
 \end{equation}
\item We say that $(\M,\tau,\a)$ enjoys \emph{order $k$ recurrence
on a dense set} if for any non-negative $a \in \M$ with $\tau(a)
> 0$ one has
\begin{equation}\label{reco2}
 \Re \tau(a (\a^n (a)) (\a^{2n} (a)) \ldots (\a^{(k-1)n} (a))) > c > 0.
\end{equation}
for some $c>0$ and all $n$ in a set of natural numbers of positive
lower density. \item  We say that $(\M,\tau,\a_0,\ldots,\a_{k-1})$
enjoys \emph{convergence in norm} if for any $a_1,\ldots,a_{k-1}
\in \M$, the averages
$$ \frac{1}{N} \sum_{n=1}^N \a_0^{-n} ((\a_1^n (a_1)) (\a^{n}_2 (a_2)) \ldots (\a^{n}_{k-1} (a_{k-1}))) $$
converge in $L^2(\tau)$ as $N \to \infty$.
\item We say that
$(\M,\tau,\a_0,\ldots,\a_{k-1})$ enjoys \emph{weak convergence} if
for any $a_0,a_1,\ldots,a_{k-1} \in \M$, the averages
$$ \frac{1}{N} \sum_{n=1}^N \tau( (\a_0^n (a_0)) (\a_1^n (a_1)) (\a^{n}_2 (a_2)) \ldots (\a^{n}_{k-1} (a_{k-1}))) $$
converge as $N \to \infty$.
\item We say that
$(\M,\tau,\a_0,\ldots,\a_{k-1})$ enjoys \emph{recurrence on
average} if for any non-negative $a \in \M$ with $\tau(a) > 0$ one
has
\begin{equation}\label{jo}
 \liminf_{N \to \infty} \frac{1}{N} \sum_{n=1}^N \Re \tau((\a_0^n (a)) (\a^{n}_1 (a)) \ldots (\a^{n}_{k-1} (a))) > 0.
\end{equation}
\item We say that $(\M,\tau,\a)$ enjoys \emph{order $k$ recurrence
on a dense set} if for any non-negative $a \in \M$ with $\tau(a) >
0$ one has
\begin{equation}\label{jo2}
 \Re \tau((\a_0^n (a)) (\a^{n}_1 (a)) \ldots (\a^{n}_{k-1} (a))) > c > 0.
 \end{equation}
for some $c>0$ and all $n$ in a set of natural numbers of positive lower density.
\end{itemize}
\end{definition}

\begin{remark}  As before, we may normalise $\alpha_0$ to be the identity. Of course, the first four properties
here are nothing more than the specialisations of the last four to
the case $\alpha_i = \alpha^i$ for $0 \leq i \leq k-1$. The real
part is needed in \eqref{reco}, \eqref{reco2}, \eqref{jo},
\eqref{jo2} because there is no necessity for the traces here to
be real-valued (the difficulty being that the product of two
non-negative elements of a non-abelian von Neumann algebra need
not remain non-negative).  In the case of \eqref{reco}, one can
omit the real part by taking averages from $-N$ to $N$, since one
has the symmetry
\begin{align*}
\overline{ \tau(a (\a^n (a)) (\a^{2n} (a)) \ldots (\a^{(k-1)n}
(a))) } &=
 \tau( (a (\a^n (a)) (\a^{2n} (a)) \ldots (\a^{(k-1)n} (a)))^*) \\
 &= \tau( (\a^{(k-1)n} (a)) \ldots (\a^{2n} (a)) (\a^n (a)) a ) \\
 &= \tau( a (\a^{-n} (a)) \ldots (\a^{-(k-1)n} (a)) )
\end{align*}
for any self-adjoint $a$.

Note however that it is quite possible for the expressions
\eqref{reco2}, \eqref{jo2} to be negative even when $a$ is
non-negative.  Because of this, while recurrence on average still
implies recurrence on a dense set, the converse is not true; one can
have recurrence on a dense set but end up with a zero or even
negative average due to the presence of large negative values of
\eqref{reco2} or \eqref{jo2}. We will see examples of this later in
this paper. \fin
\end{remark}

\begin{remark} As mentioned earlier, the Furstenberg correspondence principle equates recurrence results with a
combinatorial statements (such as Szemer\'edi's theorem) which can
be formulated in a purely finitary fashion. However, we do not know
whether the same is true for non-commutative recurrence results.
Formulating a finitary statement that would imply recurrence results
for some non-abelian von Neumann dynamical system probably requires
some quite strong approximate embeddability of the system into
finite-dimensional matrix algebras with approximate shifts, together
with a recurrence assertion for such finite-dimensional systems in
which the various parameters may all be chosen independent of the
dimension.  Since many of the results we prove below in the
infinitary setting are negative anyway, we will not pursue this
issue here.\fin
\end{remark}

The study of these properties (and related topics) for von Neumann
dynamical systems has been pursued by Niculescu, Str\"oh and
Zsid\'o~\cite{NicStrZsi03}, Duvenhage~\cite{Duv09}, Beyers,
Duvenhage and Str\"oh~\cite{BeyDuvStr07}, and Fidaleo \cite{fidaleo28}.  A variant of these questions, in which one averages over a higher-dimensional range of shifts, was also studied in \cite{fidaleo27}.  In this paper we shall
develop further positive and negative results regarding these
properties, which we now present.

\subsection{Positive results}

We first remark that when $k=2$, all systems enjoy norm and weak
convergence, as well as recurrence on average and on a dense set,
thanks to the ergodic theorem for von Neumann algebras (see e.g.
\cite[Section 9.1]{Kre85}).  Indeed, from that theorem, we know that
for any von Neumann dynamical system $(\M,\tau,\a)$ and $a \in \M$,
the averages $\frac{1}{N} \sum_{n=1}^N \a^n(a)$ converge in
$L^2(\tau)$ to the orthogonal projection of $a$ to the invariant
space $L^2(\tau)^\a := \{ f \in L^2(\tau): \a (f) = f \}$, giving
the convergence results. If $a$ is non-negative and non-zero, this
projection can be verified to have a positive inner product with
$a$, giving the recurrence results.

Now we consider the cases $k \geq 3$.  We have already seen from
Theorems \ref{comm}, \ref{comm-multid} that we have convergence and
recurrence in those abelian systems arising from ergodic theory, and
have recalled above that in fact these include all examples (up to
isomorphism).

\begin{proposition}\label{thm:ab-aves}  Let $k \geq 2$.
If $(\M, \tau, \a)$ is an abelian von Neumann dynamical system,
then $(\M,\tau,\a)$ enjoys weak convergence and convergence in
norm, and recurrence on average and on a dense set.

More generally, if $(\M,\tau,\a_0,\ldots,\a_{k-1})$ is an abelian
von Neumann $\bbZ^k$-system, then this $\bbZ^k$-system enjoys weak
convergence and convergence in norm, and recurrence on average and
on a dense set.
\end{proposition}

We now generalise these results to the wider class of
\emph{asymptotically abelian} systems.

\begin{dfn}[Asymptotic abelianness]\label{dfn:asympAb}
A von Neumann dynamical system $(\M,\tau,\a)$ is
\emph{asymptotically abelian} if one has
$$ \lim_{N \to \infty} \frac{1}{N}\sum_{n=1}^N \|[\a^n(a),b]\|_{L^2(\tau)} = 0$$
for all $a,b \in \M$, where $[a,b] := ab-ba$ is the commutator.
\end{dfn}

\begin{remark} In previous literature such as \cite{BeyDuvStr07}, a stronger version of asymptotic abelianness is
assumed, in which the $L^2(\tau)$ norm is replaced by the operator norm. Variants of this type of ``topological asymptotic abelianness'', and their relationship with non-commutative topological weak mixing have also been considered in \cite{kerr-li}.
%In addition, the above definition has several equivalent reformulations as shown in Proposition \ref{prop:as-abel} 
\fin
\end{remark}

Our work also singles out this case as special, since the assumption
of asymptotic abelianness seems to be essential for the correct
working of some the chief technical tools taken from the commutative
setting (particularly the van der Corput estimate). In the previous
works \cite{NicStrZsi03}, \cite{BeyDuvStr07}, \cite{Duv09},
convergence and recurrence were shown for all orders $k$ for
asymptotically abelian systems under some additional assumptions
such as weak mixing or compactness.  Our first main result shows
that in fact all asymptotically abelian systems enjoy convergence
and recurrence.

\begin{theorem}\label{thm:asympAb-aves}  Let $k \geq 2$.
If $(\M, \tau, \a)$ is an asymptotically abelian von Neumann
dynamical system, then $(\M,\tau,\a)$ enjoys weak convergence and
convergence in norm, and recurrence on average and on a dense set.

More generally, if $(\M,\tau,\a_0,\ldots,\a_{k-1})$ is a von
Neumann $\bbZ^k$-system, and the $\a_i \a_j^{-1}$ for $i \neq j$
are each individually asymptotically abelian, then this
$\bbZ^k$-system enjoys weak convergence and convergence in norm,
and recurrence on average and on a dense set.
\end{theorem}

Theorem \ref{thm:asympAb-aves} is deduced from the genuinely
abelian case (Proposition \ref{thm:ab-aves}) using two results.
The first is essentially from \cite{BeyDuvStr07} or \cite{Duv09},
which considered the model case $\a_i = \a^i$; for the sake of
completeness, we present a proof in Appendix \ref{duv}.

\begin{theorem}[Multiple ergodic averages for relatively weakly mixing extensions]\label{thm:rel-w-m-red} Let
$(\M,\tau,\a_0,\ldots,\a_{k-1})$ be a von Neumann $\bbZ^k$-system,
and let $\N$ be a von Neumann subalgebra of $\M$ which is invariant
under all of the $\a_i$.  If for any distinct $0 \leq i,j \leq k-1$
the shift $\a_i\a_j^{-1}$ is asymptotically abelian and weakly
mixing relative to $\N$, then the associated multiple ergodic
averages satisfy
\[\Big\|\frac{1}{N}\sum_{n=1}^N \a_0^{-n} \prod_{i=1}^{k-1} \a_i^n(a_i) -
\frac{1}{N}\sum_{n=1}^N \a_0^{-n} \prod_{i=1}^{k-1}
\a_i^n(E_\N(a_i))\Big\|_{L^2(\tau)} \to 0\] as $N\to\infty$, where $E_\N:\M\to\N$
is the conditional expectation constructed from $\tau$, and the products are from left to right.
\end{theorem}

We will recall the notions of relative weak mixing and conditional expectation in Section \ref{sec:FurZim}.

The second result, which is new and may have other applications
elsewhere, can be viewed as a partial analogue of the
Furstenberg-Zimmer structure theorem~\cite{FKO} for asymptotically
abelian systems.

\begin{theorem}[Structure theorem for asymptotically abelian systems]\label{thm:rel-w-m-centre}
If $(\M,\tau,\alpha)$ is an asymptotically abelian von
Neumann dynamical system, then $\a$ is weakly mixing relative to the
centre $\Z(\M)\subset \M$.
\end{theorem}

\begin{remark} In the case when $\M$ is a factor (i.e. the centre is trivial), results of this nature (with a slightly different notion of mixing, and of asymptotic abelianness) was established in \cite[Example 4.3.24]{bratteli}.
\end{remark}

These results quickly imply Theorem \ref{thm:asympAb-aves}.  Indeed,
when studying (for instance) convergence in norm for a
$\bbZ^k$-system, one can use Theorem \ref{thm:rel-w-m-centre}
followed by Theorem \ref{thm:rel-w-m-red} to replace each of the
$a_0,\ldots,a_{k-1}$ by their conditional expectations $E_{\Z(\M)}
(a_0),\ldots,E_{\Z(\M)} (a_{k-1})$ without any affect on the
convergence, at which point one can apply Proposition
\ref{thm:ab-aves}.  (Note that the centre $\Z(\M)$ does not depend
on what shift $\alpha_i^{-1} \alpha_j$ one is analysing.)  The other
claims are similar (using Lemma \ref{lem:condexp} to ensure that if
$a$ is non-negative with positive trace, then so is the conditional
expectation $E_{\Z(\M)} (a)$).

\begin{remark} The above arguments in fact show a more quantitative statement: if $a$ is non-negative with
$\|a\| \leq 1$ and $\tau(a) \geq \delta$ for some $0 \leq \delta
\leq 1$, then one has the same lower bound $c(k,\delta) \geq 0$ for
\eqref{reco2} as is given by Szemer\'edi's theorem for \eqref{limf}
for non-negative functions $f$ with $\|f\|_{L^\infty(X)} \leq 1$ and
$\int_X f\ d\mu \geq \delta$ (in particular, one could insert the
bound of Gowers~\cite{Gow01}). Similar remarks apply to multiple
commuting shifts.  We leave the details to the reader.\fin
\end{remark}

The proof of Theorem~\ref{thm:rel-w-m-centre}, given in
Section~\ref{sec:FurZim} below, rests on non-commutative versions of
several of the steps on the way to the Furstenberg-Zimmer Structure
Theorem in the commutative world of ergodic
theory~\cite{Fur77,Zim76.1,Zim76.2}. In particular, it rests on a
version of the dichotomy between relatively weakly mixing inclusions
and those containing a relatively isometric subinclusion, well-known
in ergodic theory from the work of Furstenberg~\cite{Fur77} and
Zimmer~\cite{Zim76.1,Zim76.2} and already generalized to the
non-commutative world by Popa in~\cite{Pop07}, for applications to
the study of superrigidity phenomena.

If $(\M,\tau,\a)$ is not asymptotically abelian then matters are
rather more complicated, with positive results only obtaining
under additional restrictions.  For $k=3$ and for ergodic shifts,
we have a positive result, established in Section \ref{triple}:

\begin{theorem}\label{thm:gen-triple-aves}  If $k=3$ and $(\M,\tau,\a)$ is an ergodic von Neumann dynamical system,
then one has weak convergence and convergence in norm, as well as
recurrence on a dense set.
\end{theorem}

We remark that the weak convergence result was previously established in \cite{fidaleo28}.

%In fact we show that recurrence will occur on almost all of a Bohr set, see Remark \ref{bohr}.

\subsection{Negative results}

Recurrence on average has been omitted from Theorem
\ref{thm:gen-triple-aves}.  This is because this result fails:

\begin{theorem}\label{ergfail}  Let $k=3$, then there exists an ergodic von Neumann dynamical system $(\M,\tau,\a)$
for which recurrence on average fails.  (In fact one can make the average \eqref{reco} strictly negative.)
\end{theorem}

We establish this in Section \ref{3fail-sec}.  The main tool is a
sophisticated version of the Behrend set construction, combined
with the crossed product construction.

When one drops the ergodicity assumption\footnote{In the commutative case, an easy application of the ergodic decomposition allows one to recover the non-ergodic case of the recurrence and convergence results from the ergodic case.  Unfortunately, in the non-commutative case, the ergodic decomposition is only available when the invariant factor $\M^\tau$ is central, which is the case in the asymptotically abelian case, but not in general.}, one also loses recurrence on a dense set:

\begin{theorem}\label{3fail}  Let $k=3$, then there exists a von Neumann dynamical system $(\M,\tau,\a)$ for
which recurrence on a dense set fails.  (In fact one can make the
means \eqref{reco2} equal to a negative constant for all non-zero
$n$.)
\end{theorem}

We establish this in Section \ref{3fail-sec} also.  This result is
simpler to prove than Theorem \ref{ergfail}, and uses the original
Behrend set construction, and crossed product constructions.

One also loses recurrence on a dense set for larger $k$ even when ergodicity is assumed:

\begin{theorem}\label{5fail}  Let $k \geq 5$ be odd, then there exists an ergodic von Neumann dynamical system
$(\M,\tau,\a)$ for which recurrence on a dense set fails.  (In
fact one can make the means \eqref{reco2} equal to a negative
constant for all non-zero $n$.)
\end{theorem}

We establish this in Section \ref{5fail-sec}.  This result uses a
counterexample of Bergelson, Host, Kra, and Ruzsa \cite{bhk},
combined with a group theoretic construction.  The restriction to
odd $k$ is mostly technical and can almost certainly be removed;
however, we are unable to decide whether Theorem \ref{5fail} can
be extended to the $k=4$ case, because it was shown in \cite{bhk}
that the $k=5$ counterexample in that paper cannot be replicated
for $k=4$.

For convergence, we have counterexamples for $k \geq 4$ even when assuming ergodicity:

\begin{theorem}\label{4fail}  Let $k \geq 4$, then there exists an ergodic von Neumann dynamical system
$(\M,\tau,\a)$ for which weak convergence and convergence in norm
fail.
\end{theorem}

We establish this in Section \ref{4fail-sec}.  The main tool is a group theoretic construction.

The above counterexamples were for the single shift case, but of
course they are also counterexamples to the more general situation
of multiple commuting shifts.  We summarise the positive and
negative results (in the single shift case) in Table
\ref{convtable}.

\begin{table}[ht]
\caption{Positive and negative results for non-commutative
convergence and recurrence of a single shift for various values of
$k$, and for various assumptions of ergodicity.  The entries
marked ``No?'' would be expected to have a negative answer if one
adopts the principle that recurrence results which fail for one
value of $k$, should also fail for higher values of
$k$.}\label{convtable}
\begin{tabular}{|l|l|l|l|l|}
\hline
 & Conv. norm? & Conv. mean? & Recur. avg.? & Recur. dense? \\
\hline
$k=2$       & Yes & Yes & Yes & Yes \\
\hline
$k=3$, erg. & Yes & Yes & No & Yes \\
$k=3$, non-erg. & ??? & ??? & No & No \\
\hline
$k\geq 4$, even, erg. & No & No & No? & ??? \\
$k\geq 4$, even, non-erg. & No & No & No? & No? \\
\hline
$k\geq 5$, odd, erg. & No & No & No & No \\
$k\geq 5$, odd, non-erg. & No & No & No & No \\
\hline
\end{tabular}
\end{table}

We note in particular that the following questions remain open:

\begin{problem}\label{silly}  If $k=3$, does weak or norm convergence hold for non-ergodic von Neumann
dynamical systems $(\M,\tau,\a)$?
\end{problem}

\begin{problem}\label{sillier}  If $k=3$, does weak or norm convergence hold for von Neumann
$\bbZ^3$-systems $(\M,\tau,\a_0,\a_1,\a_2)$ (possibly after
imposing suitable ergodicity hypotheses)?
\end{problem}

\begin{problem}  If $k=4$ (or if $k \geq 6$ is even), does recurrence on a dense set hold for ergodic
von Neumann dynamical systems $(\M,\tau,\a)$?
\end{problem}

We present some remarks on the first two problems in Section \ref{clos-sec}.

\textbf{Notational remark.}\quad Unfortunately this paper stands
between two quite unrelated uses of the word `factor', one from
operator algebras and one from ergodic theory.  In the hope that it
may be of interest to operator algebraists, we have deferred to
their usage (even though the true notion of a factor due to Murray
and von Neumann is actually not essential to our work), and will
refer throughout to inclusions of von Neumann algebras, even in the
commutative setting where these can be identified with
ergodic-theoretic `factors'. \fin

\subsubsection*{Acknowledgements} Our thanks go to Sorin Popa for
several helpful discussions, Francesco Fidaleo and David Kerr for references, and to Ezra Getzler for explaining Grothendieck's interpretation of a group via its sheaf of flat connections.  The authors are indebted to the anonymous referee for careful comments and suggestions.
\address{Brown University \and Universit\"at  T\"ubingen \and University of California, Los Angeles.

\section{Counterexamples}\label{countersec}

In this section we construct various counterexamples of von
Neumann systems $(\M, \tau, \a)$ which will demonstrate the
negative results in Theorems \ref{ergfail}-\ref{4fail}.  The
material in this section is independent of the positive results in
the rest of the paper, but may provide some cautionary intuition
to keep in mind when reading the proofs of those results.

\subsection{Non-convergence for $k \geq 4$}\label{4fail-sec}

We first show that convergence results fail for $k \geq 4$, even
if one assumes ergodicity. In fact the divergence is so bad that
it is essentially arbitrary:

\begin{theorem}[No convergence for $k \geq 4$]\label{k4}  Let $k \geq 4$ be an integer, and let $A \subset \bbZ$
be a set.  Then there exist an ergodic von Neumann system
$(\M,\tau,\a)$ and elements $a_0,\ldots,a_{k-1} \in \M$ such that
$$ \tau( a_0 \a^n (a_1) \ldots \a^{(k-1)n}( a_{k-1} ) ) = 1_A(n)$$ for all integers $n$.
\end{theorem}

It is clear that this implies Theorem \ref{4fail} by choosing $A$
appropriately (and noting that failure of weak convergence implies
failure of convergence in norm, by Cauchy-Schwarz applied in the
contrapositive).

\begin{proof}  It will suffice to verify the $k=4$ case, as the higher cases follow by setting
$a_j = 1$ for $j \geq 4$.  We will need a group $G$ with four
distinguished elements $e_0,e_1,e_2,e_3$ and an automorphism $T: G
\to G$ such that $T^k$ has no fixed points other than the identity
for all $k \neq 0$, and such that
$$
e_0 (T^{r} e_1) (T^{2r} e_2) (T^{3r} e_3) = \id
$$
holds for all $r \in A$ and fails for all $r \in \bbZ \backslash A$.
The construction of such a group is somewhat non-trivial and is
deferred to Appendix \ref{groupthy}, and in particular to
Proposition \ref{apthy}.

The group algebra $\bbC G$ of formal finite linear combinations of
group elements of $G$, acts (on the left) on the Hilbert space
$\ell^2(G)$ in the obvious manner (arising from convolution on $G$),
and can thus be viewed as a subspace of the von Neumann algebra $B(
\ell^2(G) )$ (note that all the elements of $G$ become unitary in
this perspective).  We can place a finite faithful trace $\tau$ on
$\bbC G$ by declaring the identity element to have trace $1$, and
all other elements of $G$ to have trace zero. If we then define $\M$
to be the closure of $\bbC G$ in the weak operator topology of $B(
\ell^2(G) )$, we obtain a finite von Neumann algebra, known as the
\emph{group von Neumann algebra} $LG$ of $G$. The shift $T$ leads to
an algebra isomorphism $\a$ of $\bbC G$, which then easily extends
to a shift $\a$ on $\M = LG$. Because none of the powers of $T$ have
any non-trivial fixed points, the orbit of any non-zero group
element contains no repetitions, and so one can easily establish
that $\a^n f$ converges weakly to $\tau(f)$ as $n \to \infty$ for
every $f \in \bbC G$, and hence by approximation that the unitary
operator on $\ell^2(G)$ associated to $\a$ has no fixed points
outside $\bbC\delta_\id$. This implies that $(\M,\tau,\alpha)$ is
ergodic, since given $a \in \M$ for which $\a(a) = a$ and $\tau(a) =
0$ it follows that $a(\delta_\id) \in \ell^2(G)$ is a fixed point for
the action of $T$ on $\ell^2(G)$, which must therefore equal
$\tau(a)\delta_\id = 0$, and hence $\tau(a^\ast a) =
\|a(\delta_\id)\|_2^2 = 0$ and so $a= 0$, by the faithfulness of
$\tau$. If we now set $a_j = e_j$ for $j=0,1,2,3$ we obtain the
claim.
\end{proof}

\begin{remark} An inspection of the proofs of Proposition \ref{k4} and Proposition \ref{apthy} shows
that the expression $a_0 \a^n (a_1) \a^{2n} (a_2) \a^{3n} (a_3)$
can more generally be replaced by $\a^{c_0 n} (a_0) \a^{c_1 n}
(a_1) \a^{c_2 n} (a_2) \a^{c_3 n} (a_3)$ whenever
$c_0,c_1,c_2,c_3$ are integers with $c_i \neq c_{i+1}$ for all
$i=0,1,2,3$ (with the cyclic convention $c_{i+4}=c_i$).  Thus for
instance one can construct von Neumann systems for which
$$ \tau( a_0 (\a^n (a_1)) a_2 \a^n (a_3) ) = 1_A(n)$$
 for an arbitrary set $A$. We omit the details. \fin
\end{remark}

\begin{remark} The examples of non-convergence given above are not self-adjoint or positive,
and the $a_i$ are not equal to each other.  However, it is not
hard to modify the examples to give an example of a positive $a_i
= a$ for which the averages $\frac{1}{N} \sum_{n=1}^N \tau( a \a^n
(a) \a^{2n} (a) \a^{3n} (a) )$ do not converge.  Indeed, one can
repeat the above construction with
$$ a := \id + \frac{1}{100} \sum_{i=0}^3 (e_i + e_i^*);$$
this is easily seen to be positive and self-adjoint, and a
modification of the above computations then shows that
$$ \tau( a \a^n (a) \a^{2n} (a) \a^{3n} (a) ) = 1 + \frac{2}{100^4} 1_A(n)$$
for all $n$, which is enough to ensure divergence by choosing $A$
appropriately.  We leave the details to the reader. \fin
\end{remark}

\begin{remark} The group $G$ constructed here can easily be shown to have infinite
conjugacy classes (by the same methods used to prove Proposition
\ref{apthy}).  This implies that the group algebra $LG$ is a factor.
We refer to Kadison, Ringrose \cite[Theorem 6.7.5]{KadRin97} for
details.\fin
\end{remark}

\subsection{Negative averages for $k=3$}\label{3fail-sec}

We now show the negativity of various triple averages.  The main
tool is the following Behrend-type construction of a set which
avoids progressions of length three, but contains many
``hexagons'':

\begin{lemma}[Behrend-type example]\label{beh}  Let $\eps > 0$.  Then for all sufficiently large $d$,
there exists a subset $F$ of $\bbZ/d\bbZ$ such that $|F| \geq
d^{1-\eps}$, but $F$ contains no non-trivial arithmetic
progressions of length three, thus $n,n+r,n+2r \in F$ can only
occur if $r=0$.  On the other hand, the set
$$ \{ (x,h,k) \in \bbZ/d\bbZ: x,x+h,x+k,x+k+2h, x+2k+h, x+2k+2h \in F \}$$
of ``hexagons'' in $F$ has cardinality at least $d^{3-\eps}$.
\end{lemma}

\begin{figure}[tb]
\centerline{\includegraphics{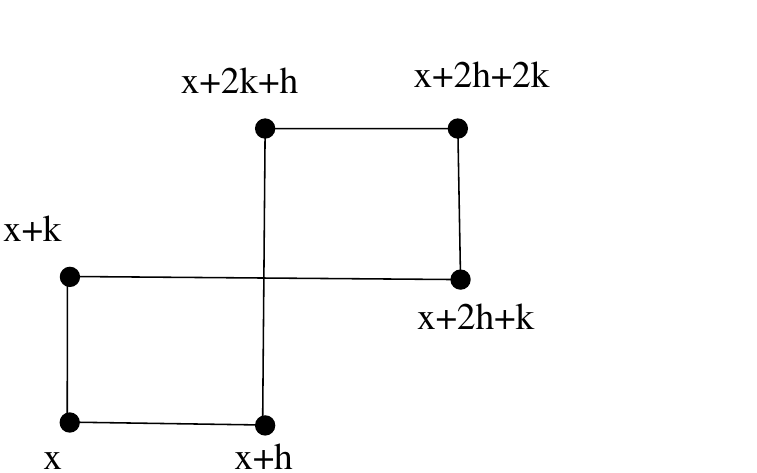}} \caption{A hexagon.  Note
the absence of arithmetic progressions of length three.}
\label{hexagon-fig}
\end{figure}

We remark that the first part of the lemma already follows directly
from the work of Behrend \cite{behrend} or the earlier work of Salem
and Spencer \cite{salem-spencer}.  The claim about hexagons will be
needed in the proof of Theorem \ref{negav2} below, but is not needed
for the simpler results in Corollary \ref{negns} or
Theorem~\ref{negtrace}.

\begin{proof} Let $R$ be a large multiple of $400$ (depending on $\eps$).  We claim that for $n$
a large enough multiple of $4$ (depending on $R$), the set
$\{-R,\ldots,R\}^n \subset \bbZ^n$ contains a subset $E$ of
cardinality $|E| \geq e^{-O(n)} R^n$ (where the implied constant in
the $O()$ notation is absolute), and which contains $\geq e^{-O(n)}
R^{3n}$ hexagons $\{ x, x+h, x+k, x+k+2h, x+2k+h, x+2k+2h\}$ but
contains no arithmetic progressions of length three. Choosing $d$
sufficiently large, letting $n$ be the largest integer such that
$(10 R)^n \leq d$ and then embedding $\{-R,\ldots,R\}^n$ in
$\bbZ/d\bbZ$ using base $10R$ (say), as in the work of Behrend or
Salem-Spencer, this claim will imply the lemma (choosing $R$
sufficiently large depending on $\eps$).

It remains to establish the claim. From the classical results on the
Waring problem (see e.g. \cite{vaughan-book}), we know that every large integer $N$ has $\sim
N^{(k-2)/2}$ representations as the sum of $k$ squares for $k$ large
enough (one can for instance take $k=5$, but for our purposes any
fixed $k$ will suffice).  Using this, we see that for any fixed
$\delta \in (0,\frac{1}{10})$, every integer $r$ such that $\delta
R^2 n \leq r \leq \frac{1}{10} R^2 n$ (say) will have $\geq
(c_\delta R)^{n - C_\delta}$ representations as the sum of $n$
squares of integers less than $R$, where $c_\delta, C_\delta > 0$
depend only on $\delta$.  In other words, the sphere $E_r := \{ x
\in \{-R,\ldots,R\}^n: |x|^2 = r \}$ has cardinality at least
$(c_\delta R)^{n - C_\delta}$.  On the other hand, such spheres have
no non-trivial progressions of length three.  Thus it will suffice
(for $n$ large enough) by the pigeonhole principle to show that
there are at least $e^{-O(n)} R^{3n}$ hexagons $\{ x, x+h, x+k,
x+k+2h, x+2k+h, x+2k+2h\}$ in $\{-R,\ldots,R\}^n$ such that
\begin{equation}\label{sol1}
 |x|^2 = |x+h|^2 = |x+k|^2 = |x+k+2h|^2 = |x+2k+h|^2 = |x+2k+2h|^2 \leq \frac{1}{10} R^2 n
\end{equation}
(note that the case when $|x|^2 \leq \delta R^2 n$ for
sufficiently small $\delta$ can be eliminated by crude estimates).

To count the solutions to \eqref{sol1}, we perform some elementary changes of variable to replace the constraints in \eqref{sol1} with simpler constraints.  We begin by observing that if $a,b,c \in \{-R/100,\ldots,R/100\}^n$ are such that
\begin{equation}\label{abc}
 a \cdot b = b \cdot c = c \cdot a = 0; \quad c \cdot c = 3 b \cdot b
\end{equation}
then $x := a-2b$, $h := b+c$, $k := b-c$ can be verified to be a
solution to \eqref{sol1}, with the map $(a,b,c) \to (x,h,k)$ being
injective, so it suffices to show that there are at least
$e^{-O(n)} R^{3n}$ triples $(a,b,c)$ with the above properties.

For reasons that will become clearer later, we will initially work in dimension $n/4$ rather than $n$.
Using the Waring problem results as before, we can find at least
$e^{-O(n)} R^{3n/4}$ triples $a,b,c \in
\{-R/400,\ldots,R/400\}^{n/4}$ such that
$$ c \cdot c = 3 b \cdot b.$$
This is one of the four constraints required for \eqref{abc}.  To obtain the remaining constraints, we use a pigeonholing trick followed by a tensor power trick.  Firstly, observe that whenever $a,b,c \in \{-R/400,\ldots,R/400\}^{n/4}$, then $a \cdot b, b \cdot c, c \cdot a$ are of order $O(R^2 n) \leq e^{O(n)}$.  Applying the pigeonhole principle, one can thus find $h_1,h_2,h_3 = O(
R^2 n )$ such that there are $e^{-O(n)} R^{3n/4}$ triples $a,b,c
\in \{-R/400,\ldots,R/400\}^{n/4}$ with
\begin{equation}\label{ab0}
a \cdot b = h_1; \quad b \cdot c = h_2; \quad c \cdot a = h_3; \quad c \cdot c = 3 b \cdot b.
\end{equation}
This is an inhomogeneous version of \eqref{abc} (at dimension $n/4$ rather than $n$), with the zero coefficients replaced by more general coefficients $h_1,h_2,h_3$.  To eliminate these coefficients we use a tensor power trick.  Let $S \subset \{-R/400,\ldots,R/400\}^{n/4} \times \{-R/400,\ldots,R/400\}^{n/4} \times \{-R/400,\ldots,R/400\}^{n/4}$ be the set of all triples $(a,b,c)$ obeying \eqref{ab0}.  We then observe that if $(a_i,b_i,c_i) \in S$ for $i=1,2,3,4$, then the vectors $a,b,c \in \bbZ^n$ defined by
$$ a := (a_1,a_2,a_3,a_4); \quad b := (b_1,b_2,-b_3,-b_4); \quad c := (c_1,-c_2,c_3,-c_4)$$
solve \eqref{abc}.  The map from the $(a_i,b_i,c_i)$ to $(a,b,c)$ is an injection from $S^4$ to the solution set of \eqref{abc}, and so we obtain at least $|S|^4 \geq e^{-O(n)} R^{3n}$ solutions to \eqref{abc} as required.
\end{proof}

This leads to a useful matrix counterexample:

\begin{lemma}[Restricted third moment can be negative]\label{negav}  There exists a positive semi-definite
Hermitian matrix $(A(j,k))_{1 \leq j,k \leq d}$ for which the
quantity
\begin{equation}\label{nrd}
 \sum_{n,r \in \bbZ/d\bbZ} A(n,n+r) A(n+r,n+2r) A(n+2r,n)
\end{equation}
is negative, where we extend $A(i,j)$ periodically in both variables by $d$.
\end{lemma}

\begin{proof}  We will take $d$ to be a multiple of $3$, and $A(j,k)$ to take the form
$$ A(j,k) := 1_E(j) 1_E(k) + 1_E(j) \omega^{-j} 1_E(k) \omega^{k}$$
where $E \subset \bbZ/d\bbZ$ is a set to be determined later, and
$\omega := e^{2\pi i/3}$ is a cube root of unity.  The matrix
$(A(j,k))_{1 \leq j, k \leq d}$ is then the sum of two rank one
projections and is thus positive semi-definite and Hermitian.  The
expression \eqref{nrd} can be expanded as
$$
\sum_{n,r \in \bbZ/d\bbZ: n,n+r,n+2r \in E} (1 + \omega^r) (1 +
\omega^r) (1 + \omega^{-2r}).$$ The summand can be computed to
equal $8$ when $r$ is divisible by $3$, and $-1$ otherwise.  Thus,
to establish the claim, it suffices to find a set $E$ such that
the set
$$ \{ (n,r) \in \bbZ/d\bbZ: n,n+r,n+2r \in E; r \neq 0 \hbox{ mod } 3 \}$$
is more than eight times larger than the set
$$ \{ (n,r) \in \bbZ/d\bbZ: n,n+r,n+2r \in E; r = 0 \hbox{ mod } 3 \},$$
thus the length three arithmetic progressions in $E$ with spacing
not divisible by $3$ need to overwhelm the length three
progressions with spacing divisible by $3$.

To do this, we use Lemma \ref{beh} to obtain a subset $F \subset
\{1,\ldots,[d/10]\}$ of cardinality $|F| \geq d^{0.99}$ which
contains no arithmetic progressions of length three.  We then pick
three random shifts $h_0, h_1,h_2 \in \{1,\ldots,d/3\}$ uniformly
at random, and consider the set
$$ E := \{ 3(f+h_i) + i: i=0,1,2; f \in F \}$$
consisting of three randomly shifted, dilated copies of $F$.

By construction, the only length three progressions in $E$ with
spacing divisible by $3$ are the trivial progressions $n, n, n$
with $r=0$, so the total number of such progressions is at most
$d$.  On the other hand, for any fixed $f_0,f_1,f_2 \in F$, the
numbers $3(f_i+h_i)+i$ for $i=0,1,2$ have a probability $3/d$ of
forming an arithmetic progression with spacing not divisible by
$3$, due to the random nature of the $h_i$.  Thus the expected
value of the total number of such progressions is at least
$(d^{0.99})^3 \times 3 / d = 3 d^{1.97}$.  For $d$ large enough,
this gives the claim.
\end{proof}

This already gives a simple example of negative averages for
non-ergodic systems:

\begin{corollary}[Negative average for non-ergodic system]\label{negns} There exists a finite von
Neumann algebra $(\M,\tau)$ with a shift $\a$, and a non-negative
element $a \in \M$, such that $\frac{1}{2N+1} \sum_{n=-N}^N \tau(
a \a^n (a) \a^{2n} (a))$ converges to a negative number.
\end{corollary}

\begin{proof}  Let $a=(A(j,k))_{1 \leq j,k \leq d}$ be as in Lemma \ref{negav}.  We let $\M$ be the
von Neumann algebra of complex $d \times d$ matrices with the
normalised trace $\tau$, and with the shift
$$ \a ( B(j,k) )_{1 \leq j,k \leq d} := ( e^{2\pi i (j-k)/d} B(j,k) )_{1 \leq j,k \leq d}.$$
This is easily verified to be a shift.
%%%
We see that
$$
 \tau(a \a^n (a) \a^{2n} (a)) = \frac{1}{d} \sum_{j,k,l\in \bbZ/d\bbZ} e^{2\pi in(k+l-2j)/d} A(j,k) A(k,l)
 A(l,j).
$$
This expression is periodic in $n$ with period $d$, and has
average
$$ \frac{1}{d} \sum_{l,r \in \bbZ/d\bbZ} A(l,l+r) A(l+r,l+2r) A(l+2r,l)$$
and the claim then follows from Lemma \ref{negav}.
\end{proof}

This shows that recurrence on average for $k=3$ can fail for
non-ergodic systems.  However, this is not yet enough to establish
either Theorem \ref{ergfail} or Theorem \ref{3fail}.  To obtain
these stronger results we must introduce the \emph{crossed product
construction} in von Neumann algebras.  For a comprehensive
introduction to this concept, see \cite[Chapter 13]{KadRin97}.  We
shall just recall the key properties of this construction we need
here.

Suppose we have a finite von Neumann algebra $(\M,\tau)$,
and an action $U$ of a (discrete) group $G$ on $\M$, thus for each
$g \in G$ we have a shift $U(g): \M \to \M$ such that $U(g) U(h) =
U(gh)$ for all $g,h \in G$, with $U(\id)$ being the identity.  Then there exists a crossed product
$(\M \rtimes_U G, \tau)$ which contains both the original space
$(\M,\tau)$ and the group algebra $\bbC G$ as subalgebras.
Furthermore, in this crossed product we have
\begin{equation}\label{gaga}
 U(g) a = g a g^{-1}
\end{equation}
for all $a \in \M$ and $g \in G$, and
$$ \tau( g a ) = \tau( a g ) = 0$$
for all $a \in \M$ and $g \in G$ with $g$ not equal to the
identity.  Finally, the span of the elements $ag$ for $a \in \M$
and $g \in G$ is dense in $\M \rtimes_U G$.

\begin{remark}\label{24} The exact construction of the crossed product is not relevant for our applications,
but for the convenience of the reader we sketch one such construction here.
We first form the Hilbert space
$$ {\mathfrak h} := \ell^2(G, L^2(\tau)) = \bigoplus_{g \in G} L^2(\tau)$$
consisting of tuples $(x_g)_{g \in G}$ in $L^2(\tau)$.  This space has an action of $\M$ defined by
$$ a (x_g)_{g \in G} := ((U(g^{-1}) a) x_g)_{g \in G}$$
for $a \in \M$, and an action of $G$ (and hence $\bbC G$) defined by
$$ h (x_g)_{g \in G} := (x_{h^{-1} g})_{g \in G}.$$
One can verify that these actions combine to an action of the
twisted convolution algebra $\ell^1( G, \M )$ on ${\mathfrak h}$,
defined as the space of formal sums $\sum_{h \in G} h a_h$ with
$\sum_{h \in G} \| a_h \| < \infty$, and subject to the relations
\eqref{gaga}.  We define a trace on such sums by the formula $\tau(
\sum_{h \in G} h a_h ) := \tau( a_{\id} )$.  One can then show that
one can extend this to a finite trace on the weak operator topology
closure of $\ell^1( G, \M )$, viewed as a subset of $B({\mathfrak
h})$; this closure can then be denoted $\M \rtimes_U G$.  In other words, 
$\M \rtimes_U G$ is constructed as the von Neumann algebra generated by the action of $\M$ and $G$ on ${\mathfrak h}$. \fin
\end{remark}

\begin{example} The group von Neumann algebra $LG$ can be viewed as $\bbC \rtimes G$, where $G$ acts
trivially on the one-dimensional von Neumann algebra $\bbC$.\fin
\end{example}

We can now get a stronger version of Corollary \ref{negns}:

\begin{theorem}[Negative trace for non-ergodic system]\label{negtrace} There exists a von Neumann dynamical
system $(\M,\tau,\a)$ and a non-negative element $a \in \M$, such
that $\tau( a \a^n (a) \a^{2n} (a))$ is negative (and independent
of $n$) for all non-zero $n$.  In particular, Theorem \ref{3fail}
holds.
\end{theorem}

\begin{proof} Let $(\M',\tau,\beta)$ be a von Neumann dynamical system to be chosen later.
Using the crossed product construction, we can build an extension
$\M := \M' \rtimes_U \bbZ^2$ of $\M'$ generated by $\M'$ and two
commuting unitary elements $u, m$, such that
\begin{equation}\label{mam}
m a m^{-1} = \beta(a)
\end{equation}
and
$$ u a u^{-1} = a$$
for all $a \in \M'$.  In particular, the element $u$ is central.
It is then easy to see that we can build\footnote{To build $\a$ explicitly, we can view $\M$ as an algebra of operators on the Hilbert space ${\mathfrak h} := \bigoplus_{(j,k) \in \bbZ^2} L^2(\tau)$ as per Remark \ref{24}, and let $\alpha$ be the conjugation $a \mapsto W a W^*$ by the unitary operator $W: {\mathfrak h} \to {\mathfrak h}$ defined by $W( x_{(j,k)} )_{(j,k) \in \bbZ^2} := (x_{(j,k-j)} )_{(j,k) \in \bbZ^2}$.} a shift $\a$ on $\M$ for
which
$$ \a(a) = a; \quad \a(u) = u; \quad \a(m) = mu$$
for all $a \in \M'$, since the action of the group $\bbZ^2$
generated by $m$ and $u$ on $\M'$ is unchanged when one replaces
$m$ by $mu$.

Now let $a \in \M$ be an element of the form
$$ a = \left(\sum_{i \in \bbZ} f_i m^i\right) \left(\sum_{i \in \bbZ} f_i m^i\right)^*$$
where $f_i \in \M'$, and only finitely many of the $f_i$ are
non-zero. This is clearly non-negative, and can be simplified by
\eqref{mam} to the power series
$$ a = \sum_{h \in \bbZ} g_h m^h$$
where the $g_h \in \M'$ are the twisted autocorrelations of the $f_j$,
$$ g_h = \sum_{j \in \bbZ} f_{j+h} \beta^h( f_j^* ).$$
Let $n$ be non-zero.  The expression $\tau( a \a^n (a) \a^{2n} (a)
)$ can be expanded as
$$ \sum_{h_1,h_2,h_3 \in \bbZ} \tau( g_{h_1} m^{h_1} g_{h_2} (mu^n)^{h_2} g_{h_3} (mu^{2n})^{h_3} ).$$
The net power of the central element $u$ here is $n(h_2+2h_3)$,
and the net power of $m$ is $h_1+h_2+h_3$.  Thus we see that the
trace vanishes unless $h_2+2h_3 = h_1+h_2+h_3 = 0$, or
equivalently if $(h_1,h_2,h_3) = (h,-2h,h)$ for some $h$.
Performing this substitution and using \eqref{mam}, we simplify
this expression to
\begin{equation}\label{has}
 \sum_{h \in \bbZ} \tau( g_h \beta^h (g_{-2h}) \beta^{-h} (g_h) ).
\end{equation}
In particular, this expression is now manifestly independent of $n
\neq 0$.

We now select $\M'$ to be the commutative von Neumann system
$L^\infty(\bbZ/d\bbZ)$ with the shift $\beta( f(x) ) := f(x+1)$
and the normalised trace.  Thus the $g_h$ and $f_h$ are now
complex-valued functions on $\bbZ/d\bbZ$, and the above expression
can be expanded explicitly as
$$ \frac{1}{d} \sum_{x \in \bbZ/d\bbZ} \sum_{h \in \bbZ} g_h(x) g_{-2h}(x+h) g_h(x-h).$$
Meanwhile, the $g_h(x)$ by definition can be written as
$$ g_h(x) = \sum_{j \in \bbZ} f_{j+h}(x) \overline{f_j(x+h)}.$$
We pick a large number $N$ to be chosen later, and set
$$ f_j(x) := b(x,x+j) 1_{1 \leq j \leq Nd}$$
where $b: \bbZ/d\bbZ \times \bbZ/d\bbZ \to \bbC$ is a function
periodic in two variables of period $d$ to be chosen later.  Then
we can compute
$$ g_h(x) = \left(1 - \frac{|h|}{dN}\right)_+ N A(x,x+h) + O(1) $$
where
\begin{equation}\label{biz}
 A(x,y) := \sum_{z \in \bbZ/d\bbZ} b(x,z) \overline{b(y,z)}
\end{equation}
and $O(1)$ denotes a quantity that can depend on $d$ (and $b$) but
is uniformly bounded in $N$.  The expression \eqref{has} can then
be computed to be
$$ C \frac{N^4}{d} \sum_{x, h \in \bbZ/d\bbZ} A(x,x+h) A(x+h,x-h) A(x-h,x) + O( N^3 )$$
where $C > 0$ is the explicit constant
$$ C := \int_{{\mathbb R}} (1-|h|)_+^2 (1-|2h|)_+\ dh.$$
By the substitution $x = m+r, h = r$, we can re-express this as
\begin{equation}\label{tao}
C \frac{N^4}{d} \sum_{m, r \in \bbZ/d\bbZ} A(m,m+r) A(m+r,m+2r) A(m+2r,m) + O( N^3 ).
\end{equation}
Now, let $d$ and $A(j,k)$ be as in Lemma \ref{negav}.  By the
spectral theorem (which in particular allows one to construct self-adjoint square roots of positive definite matrices), we can find $b(x,y)$ so that \eqref{biz} holds.
The summand in \eqref{tao} is then negative, and the claim follows
by choosing $N$ large enough depending on all other parameters.
\end{proof}

Of course, by Theorem \ref{thm:gen-triple-aves}, one cannot have
such a result when the underlying shift $\a$ is ergodic.  On the
other hand, one can extend Corollary \ref{negns} to the ergodic
case:

\begin{theorem}\label{negav2} There exists an \emph{ergodic} von Neumann system $(\M,\tau,\a)$  and a
non-negative element $a \in \M$, such that $\frac{1}{2N+1}
\sum_{n=-N}^N \tau( a \a^n (a) \a^{2n} (a))$ converges to a
negative number. In particular, Theorem \ref{ergfail} holds.
\end{theorem}

\begin{proof} Let $d$ be a large odd number, and let $u := e^{2\pi i/d}$ be a primitive $d^{th}$ root
of unity.  We will let $\M$ be a completion of the
\emph{non-commutative torus}.  This is obtained by first forming the
C$^\ast$-algebra generated by two unitary generators $e_1, e_2$
obeying the commutation relation
$$ e_2 e_1 = u e_1 e_2$$
and with all of the expressions $e_1^j e_2^k$ having zero trace
unless $j=k=0$, in which case the trace is $1$; and then completing
in the weak operator topology resulting from the
Gel'fand-Naimark-Segal representation on $L^2(\tau)$. One can
represent this finite von Neumann algebra more explicitly by
letting $e_1, e_2$ act on $L^2( (\bbR/\bbZ)^2 )$ by the maps $e_1
f(x,y) := e^{2\pi i x} f(x,y)$ and $e_2 f(x,y) := e^{2\pi i y} f(x +
1/d, y)$, with the trace $\tau$ given by $\tau(a) = \langle \Omega,
a \Omega \rangle_{L^2( (\bbR/\bbZ)^2)}$, where $\Omega \equiv 1$ is
the identity function on $(\bbR/\bbZ)^2$.

We let $\theta_1, \theta_2 \in S^1$ be generic unit phases, and
then define the shift $\a$ on $\M$ by setting
$$ \a (e_1) := \theta_1 e_1; \quad \a (e_2) := \theta_2 e_2.$$
It is easy to see that this is a shift.  If $\theta_1, \theta_2$ are
generic (so that $\theta_1^j \theta_2^k$ is not a root of unity for
any $(j,k) \neq (0,0)$), this shift is easily verified to be ergodic
(as one can verify the mean ergodic theorem by hand on the
generators $e_1^j e_2^k$, and then argue as in the proof of
Theorem~\ref{k4} using the faithfulness of $\tau$).

We set $a := gg^*$, where $g$ is an element of the form
$$ g := \sum_{k=1}^M \sum_{h \in \bbZ} c_h e_1^h e_2^k,$$
$M$ is a large number (much larger than $d$) to be chosen later, and
$c_h$ are complex numbers to be chosen later, all but finitely many
of which are zero.  Clearly $a$ is non-negative.  A computation
shows that
$$ a = \sum_{h,k \in \bbZ} c_{h,k} e_1^h e_2^k$$
where
\begin{equation}\label{dhk}
 c_{h,k} := M \left(1 - \frac{|k|}{M}\right)_+ \sum_{l \in \bbZ} c_{l+h} \overline{c_l} u^{kl}.
 \end{equation}
Since
$$ \a^n (a) = \sum_{h,k \in \bbZ} c_{h,k} \theta_1^{hn} \theta_2^{kn} e_1^h e_2^k,$$
some Fourier analysis and the genericity of $\theta_1$, $\theta_2$
show that the expression
$$ \frac{1}{2N+1} \sum_{n=-N}^N \tau( a \a^n (a) \a^{2n} (a))$$
converges as $N \to \infty$ to the expression
$$ \sum_{h,k} c_{h,k} c_{-2h,-2k} c_{h,k} \tau( e_1^h e_2^k e_1^{-2h} e_2^{-2k} e_1^h e_2^k ).$$
The trace here simplifies to $u^{3hk}$.  Inserting \eqref{dhk}, we can expand this expression as
\begin{equation}\label{mus}
 M^3 \sum_{h,k,l_1,l_2,l_3 \in \bbZ} \phi(k/M) c_{l_1+h} \overline{c_{l_1}} c_{l_2-2h}
 \overline{c_{l_2}} c_{l_3+h} \overline{c_{l_3}} u^{kl_1 - 2kl_2 + kl_3 + 3hk}
 \end{equation}
where
$$ \phi(x) := (1-|x|)_+^2 (1-|2x|)_+.$$
By Poisson summation, the expression
$$ \sum_k \phi(k/M) u^{kl_1 - 2kl_2 + kl_3 + 3hk}$$
can be computed to be $M \int_\bbR \phi(x)dx + O(1)$ if
$l_1-2l_2+l_3+3h$ is divisible by $d$, and $O(1)$ otherwise, where
$O(1)$ denotes a quantity that can depend on $d$ but is bounded
uniformly in $M$.  If we then assume that the $c_h$ vanish for $h$
outside of $\{1,\ldots,M\}$ and are bounded uniformly in $M$, we
can thus expand \eqref{mus} as
$$
C M^4 \sum_{h,l_1,l_2,l_3 \in \bbZ:\ d | l_1-2l_2 + l_3 + 3h}
c_{l_1+h} \overline{c_{l_1}} c_{l_2-2h} \overline{c_{l_2}} c_{l_3+h}
\overline{c_{l_3}}  + O( M^7 )$$ for some absolute constant $C>0$.

If we now set $c_h := b(h) 1_{[1,M]}(h)$, where $b: \bbZ/d\bbZ \to
\bbC$ is a periodic function with period $d$ and independent of
$M$ to be chosen later, we can express this as
$$ C_d M^8 \sum_{h,l_1,l_2,l_3 \in \bbZ/d\bbZ:\ l_1-2l_2+l_3+3h=0} b(l_1+h) \overline{b(l_1)} b(l_2-2h)
\overline{b(l_2)} b(l_3+h) \overline{b(l_3)} + O( M^7 )$$ for some
$C_d > 0$ depending on $d$ but independent of $M$. Making the
substitution $l_1 = x; l_2 = x+k+2h; l_3 = x+2k+h$, we see that we
will be done as soon as we are able to find $d, b$ for which the
expression
$$ X := \sum_{x,h,k \in \bbZ/d\bbZ} \overline{b(x)} b(x+h) b(x+k) \overline{b(x+k+2h)} \overline{b(x+2k+h)} b(x+2k+2h)$$
is negative.

To do this, we again appeal to Lemma \ref{beh} to find a set $F
\subset \bbZ/d\bbZ$ of size at least $d^{0.99}$ (assuming $d$
large enough), which contains no arithmetic progressions of length
three, but contains at least $d^{2.99}$ hexagons $x, x+h, x+k,
x+k+2h, x+2k+h, x+2k+2h$.  We then set
$$ b(x) := \epsilon_x 1_F(x)$$
where the $\epsilon_x = \pm 1$ are independent signs, thus $X$ is now the random variable
$$ X = \sum_{x,h,k: x,x+h,x+k,x+k+2h,x+2k+h,x+2k+2h \in F} \epsilon_x \epsilon_{x+h} \epsilon_{x+k}
\epsilon_{x+2h+k} \epsilon_{x+h+2k} \epsilon_{x+2h+2k}.$$ We will
show (for $d$ large enough) that the standard deviation of $X$
exceeds its expectation, which shows that there exists a choice of
signs for which $X$ is negative.

We first compute the expectation of $X$.  The only summands with
non-zero expectation occur when all the signs cancel, which only
occurs when $h=0$ or when $k=0$, as can be seen by an inspection
of the number of ways to collapse the hexagon in Figure
\ref{hexagon-fig}; here we need the hypothesis that $d$ is odd.
But as $F$ contains no non-trivial arithmetic progressions, there
are no summands for which only one of the $h,k$ are zero, so we
are left only with the $h=k=0$ terms, of which there are at most
$d$.  Thus the expectation of $X$ is at most $d$.

Now we compute the variance.  There are at least $d^{2.99}$
hexagons in $F$, and all but $O(d^2)$ of them are non-degenerate
in the sense that the six vertices of the hexagon are all
distinct.  The summands in $X$ corresponding to non-degenerate
hexagons have variance $1$, and the correlation between any two
summands in $X$ either zero or positive (the latter occurs when
two summands are permutations of each other).  Thus the variance
of $X$ is $\gg d^{2.99}$, so the standard deviation is $\gg
d^{1.495}$, and the claim follows.
\end{proof}

\subsection{Negative trace for $k = 5$}\label{5fail-sec}

Now we show negative traces can occur even in the ergodic case when $k = 5$.

\begin{theorem}\label{negavk} There exists an \emph{ergodic} von Neumann dynamical system $(\M,\tau,\a)$
and a non-negative element $a \in \M$, such that $\tau( a \a^n (a)
\a^{2n} (a) \a^{3n} (a) \a^{4n} (a))$ is negative for every
non-zero $n$.
\end{theorem}

This establishes the $k=5$ case of Theorem \ref{5fail}.  A similar
argument holds for all larger odd values of $k$, which we leave to
the interested reader; we restrict here to the case $k=5$ simply
for ease of notation.

To prove this theorem, our starting point is the following result of Bergelson, Host, Kra, and Ruzsa \cite{bhk}:

\begin{theorem}\label{bhkthm}  For any $\delta > 0$, there exists a measure-preserving system $(X, {\mathcal X}, \mu, S)$
and a measurable set $A \subset X$ with $0 < \mu(A) < \delta$ such
that
$$ \mu( A \cap S^n (A) \cap S^{2n} (A) \cap S^{3n} (A) \cap S^{4n} (A) ) \leq \mu(A)^{100}$$
(say) and
\begin{equation}\label{mox}
 \mu(A \cap S^n (A)) = \mu(A)^2
\end{equation}
for every non-zero integer $n$.
\end{theorem}

\begin{proof} This follows from \cite[Theorem 1.3]{bhk} (see also the remark immediately below that theorem).
The property \eqref{mox} is not explicitly stated in that theorem,
but follows from the construction in \cite[Section 2.3]{bhk} (the
system $X$ is a torus $(\bbR/\bbZ)^2$ with the skew shift $S:
(x,y) \mapsto (x+\a,y+2x+\a)$, and the set $A$ has the special
form $A = (\bbR/\bbZ) \times B$ for some set $B$).
\end{proof}

We apply this theorem for some sufficiently small $\delta$ (to be
chosen later) to obtain $X, \mu, S, A$ with the above properties.
We will combine this with the group $G$, the automorphism $T$, and
the elements $e_0,e_1,e_2,e_3,e_4$ arising from Proposition
\ref{apthy2} as follows.

First, we create the product space $L^\infty(X^G, d\mu^G)$, whose
$\s$-algebra is generated up to negligible sets by the tensor
products $\bigotimes_{g \in G} f_g$, where $f_g \in
L^\infty(X,d\mu)$ is equal to $1$ for all but finitely many $g$.
This product has a unitary, trace-preserving action $U$ of $G$,
defined by
$$ U(h) \bigotimes_{g \in G} f_g := \bigotimes_{g \in G} f_{h^{-1} g}.$$
We can therefore create the crossed product $\M := L^\infty(X^G,
d\mu^G) \rtimes_U G$.  Note that if we embed $L^\infty(X,\mu)$
into $L^\infty(X^G,d\mu^G)$ by using the identity component of
$X^G$, we have
\begin{equation}\label{slam}
\bigotimes_{g \in G} f_g = \prod_{g \in G} U(g) f_g
\end{equation}
(note that the $U(g) f_g$ necessarily commute with each other.)

We define a shift $\a$ on $\M$ by requiring that
$$ \a( \bigotimes_{g \in G} f_g ) = \bigotimes_{g \in G} S( f_{T^{-1} g} )$$
and
$$ \a(g) = Tg;$$
one can check that this is indeed a well-defined shift on $\M$.

We claim that $\a$ is ergodic.  Indeed, if $a \in \M$ is of the form
$a = f g$ for some $f \in L^\infty(X^G, d\mu^G)$ and $g \in G$ not
equal to the identity, then as the powers of $T$ have no non-trivial
fixed points, the orbit $T^n g$ escapes to infinity, and the orbit
$\a^n(a)$ converges weakly to zero.  Meanwhile, if $g$ is the
identity, then it is classical that the Bernoulli system $G
\circlearrowright L^\infty(X^G, d\mu^G)$ is ergodic, and so the
ergodic theorem applies to $a$ in this case.  Putting the two facts
together and arguing as for the ergodicity in Theorem~\ref{k4}
yields the ergodicity of $\a$.

Note that $1_A$ lies in $L^\infty(X,d\mu)$, and can thus be
identified with an element of $\M$ by the previous embedding.  We
set
$$ a := \sum_{i=0}^3 1_A \cdot (2 - e_i - e_i^{-1}) \cdot 1_A.$$
Clearly $a$ is non-negative.  Now let $n$ be non-zero, and consider
the expression
\begin{equation}\label{atau}
\tau( a \a^n (a) \a^{2n} (a) \a^{3n} (a) \a^{4n} (a)).
\end{equation}
Expanding out $a$, we obtain a linear combination of terms of the form
$$
\tau( 1_A g_0 1_A 1_{S^n (A)} (T^n g_1) 1_{S^n (A)} 1_{S^{2n} (A)}
(T^{2n} g_2) 1_{S^{2n} (A)} 1_{S^{3n} (A)} (T^{3n} g_3) 1_{S^{3n}
(A)} 1_{S^{4n} (A)} (T^{4n} g_4) 1_{S^{4n} (A)} )$$ where
$$g_0,g_1,g_2,g_3,g_4 \in \{ id, e_0, e_1, e_2, e_3, e_4, e_0^{-1}, e_1^{-1}, e_2^{-1}, e_3^{-1}, e_4^{-1} \}.$$
This trace vanishes unless
\begin{equation}\label{gosh}
 g_0 T^n g_1 T^{2n} g_2 T^{3n} g_3 T^{4n} g_4 = \id.
\end{equation}
By Proposition \ref{apthy2}, we conclude that
$g_0,g_1,g_2,g_3,g_4$ are either all equal to the identity, or are
a permutation of $e_0, e_1, e_2, e_3, e_4$, or are a permutation
of $e_0^{-1}, e_1^{-1}, e_2^{-1}, e_3^{-1}, e_4^{-1}$.  In the
latter two cases, the contribution to \eqref{atau} is either zero
or negative (being negative the trace of the product of several
non-negative elements in a commutative von Neumann algebra).  Here
we are using the fact that $5$ is odd.  Discarding all of these
contributions except the one where $g_{i,0} = e_{i,0}$ (which has
a non-trivial contribution thanks to Proposition \ref{apthy2}), we
conclude that \eqref{atau} is at most
\begin{align*}
&10^{5} \tau( 1_A 1_{S^n (A)} 1_{S^{2n} (A)} 1_{S^{3n} (A)} 1_{S^{4n} (A)} ) \\
\quad & - \tau( 1_A e_{0} 1_A 1_{S^n (A)} e_{1} 1_{S^n (A)}
1_{S^{2n} (A)} e_{2} 1_{S^{2n} (A)} 1_{S^{3n} (A)} e_{3} 1_{S^{3n}
(A)} 1_{S^{4n} (A)} e_{4} 1_{S^{4n} (A)} ).
\end{align*}

By Theorem \ref{bhkthm}, the first expression is at most $10^{5}
\mu(A)^{100}$.  Now consider the second expression.  By
Proposition \ref{apthy2}, we see that the partial products $e_{0}
e_{1} \ldots e_{i}$ for $i=0,1,2,3$ are distinct.  Using
\eqref{slam}, we conclude that the trace here can be computed as
$$ \mu( S^{4n} (A) \cap A ) \mu( A \cap S^n (A) ) \mu( S^n (A) \cap S^{2n} (A) ) \mu( S^{2n} (A) \cap S^{3n} (A))
\mu( S^{3n} (A) \cap S^{4n} (A) ),$$ which by \eqref{mox} is equal
to $\mu(A)^{10}$.  Thus the expression \eqref{atau} is at most
$2^{15} \mu(A)^{100} - \mu(A)^{10}$, which is negative if the
upper bound $\delta$ for $\mu(A)$ is chosen to be sufficiently
small.

This concludes the proof of Theorem \ref{negavk}.

\begin{remark} Given that the counterexample in Theorem \ref{bhkthm} can be extended to any $k \geq 5$,
it seems reasonable to expect that Theorem \ref{5fail} can be
extended to all $k \geq 5$ (not just the odd $k$), though we have
not pursued this issue.  On the other hand, the analogue of
Theorem \ref{bhkthm} fails for $k=4$, as was shown in \cite{bhk}.
Because of this, the $k=4$ case of Theorem \ref{5fail} remains
open; the construction given here does not work, but it is
possible that some other construction would suffice instead.\fin
\end{remark}

\section{Inclusions of finite von Neumann dynamical
systems}\label{sec:FurZim}

In this section we quickly recall some fairly well-known
constructions relating to von Neumann dynamical systems and their
basic properties, culminating in a treatment of Popa's
noncommutative version of the Furstenberg-Zimmer dichotomy
from~\cite{Pop07}.  This material will be needed to establish the structure theorem (Theorem \ref{thm:rel-w-m-centre}).

Let $(\M,\tau)$ be a finite von Neumann algebra.  As noted in
the introduction, we can embed $\M$ into a Hilbert space
$L^2(\tau)$.  In order to distinguish the algebra structure from the
Hilbert space structure\footnote{It is tempting to ignore these
distinctions and identify $\hat \M$ with $\M$.  While this is
normally qutie a harmless identification, we will take some care
here because we will be studying the bi-module action of $\M$ on
$L^2(\tau)$, and keeping track of this action can become
notationally confusing if the algebra elements are identified with
the vectors that they act on.}, we shall refer in this section to
the embedded copy of an element $a \in \M$ of the algebra in
$L^2(\tau)$ as $\hat a$ rather than $a$, thus for instance $\hat \M
= \{ \hat a: a \in \M \}$ is a dense subspace of $L^2(\tau)$.

Clearly, $L^2(\tau)$ has the structure of an $\M$-bimodule, formed
by extending the regular bimodule structure on $\M$ by density; the
left-representation is, of course, the classical
Gel'fand-Naimark-Segal representation associated to $\tau$.  When it
is necessary to denote the copy of $\M$ in $B(L^2(\tau))$ consisting
of the members of $\M$ acting by multiplication on the left
(respectively, right), we will denote this algebra by
$\M_{\rm{left}}$ (respectively, $\M_{\rm{right}}$).

The space $L^2(\tau)$ contains a distinguished vector $\hat{1}$ --
the representative of the multiplicative identity $1$ in $\M$ --
with the property that $a\hat 1 = \hat 1 a = \hat a$ for all $a \in
\M$. This vector will play a prominent role in the rest of this
section.

Now let $(\N, \tau|_\N)$ be a von Neumann subalgebra of $(\M,\tau)$
(with the inherited trace).  Then we can canonically identify
$L^2(\tau|_\N)$ with the closed subspace
$$\overline{\{\hat{b}:\ b \in \N\}} = \overline{\N\hat 1} = \overline{\hat 1\N}$$
of $L^2(\tau)$ in the obvious manner.

We will make use of certain well-known properties of these
constructs, which we merely recall here.  A clear account of all of
them can be found in \cite[Chapters 1,3]{JonSun97}.

First, it is important that there is a simple necessary and
sufficient condition for a vector $\xi \in L^2(\tau)$ to lie in the
dense subspace $\widehat{\M}$: this is so if and only if the linear
operator
\[\widehat{\M} \to L^2(\tau):\hat{x} \mapsto x\xi\]
is bounded for the norm $\|\cdot\|_{L^2(\tau)}$, and so extends by
continuity to a bounded operator $L^2(\tau)\to L^2(\tau)$. The
necessity of this conclusion is clear, and its sufficiency requires
just a little argument using the fact that for a finite von Neumann algebra
$(\M,\tau)$ we have $\M_{\rm{right}} = \M_{\rm{right}}''$ and
$\M_{\rm{left}} = \M_{\rm{left}}''$; see \cite[Theorem 1.2.4]{JonSun97}.

A simple application of this condition now shows that the
orthogonal projection $e_\N:L^2(\tau) \to \overline{\N\hat 1}$
maps the dense subspace $\widehat{\M}$ into $\widehat{\N}$, and so
defines also a linear operator $E_\N:\M \to\N$. Indeed, for $a \in
\M$ we need only to show that the map
\[\widehat{\M} \to L^2(\tau):\hat{x} \mapsto xe_\N(\hat{a})\]
is bounded for the norm $\|\cdot\|_{L^2(\tau)}$.  Since $\N$ is also a von
Neumann algebra and $e_\N(\hat{a}) \in \overline{\N\hat 1}\cong
L^2(\tau|_\N)$, it actually suffices to check this for $x \in \N$.
However, since $\overline{\N\hat 1}$ is an $(\N,\N)$-sub-bimodule,
left multiplication by $x$ commutes with $e_\N$, and so we have
\[\|xe_\N(\hat{a})\|_{L^2(\tau)} = \|e_\N(x\hat{a})\|_{L^2(\tau)} \leq \|\hat{xa}\|_{L^2(\tau)}
\leq \|a\|\|\hat{x}\|_{L^2(\tau)},\]
as required.

The linear operator $E_\N$ is referred to as the \emph{conditional
expectation} of $\M$ onto $\N$ associated to $\tau$, and it has the
following readily-verified properties:

\begin{lem}[Properties of conditional expectation]\label{lem:condexp}
For all $a \in \M$, the operator $E_\N$ satisfies:
\begin{itemize}
\item (Idempotence) $E_\N(E_\N(a)) = E_\N(a)$;
\item (Contractivity) $\|E_\N(a)\| \leq \|a\|$;
\item (Trace-preservation) $\tau|_\N( E_\N(a) ) = \tau(a)$;
\item (Positivity) $E_\N(a^\ast a) \geq 0$ (as a member of
$\N$); and
\item (Relation with $e_\N$)  For all $\xi \in L^2(\tau)$, one has
\[e_\N (a (e_\N(\xi))) = E_\N(a)(e_\N(\xi)) = e_\N(E_\N(a)(\xi)).\]
\end{itemize}
\end{lem}

\begin{example} If $\M = L^\infty(X,{\mathcal X},\mu)$ for some probability measure
$\mu$ with the usual trace, and $(Y,{\mathcal Y},\nu)$ is a factor
space of $(X,{\mathcal X},\mu)$ with a measurable factor map $\pi: X
\to Y$ that pushes $\mu$ forward to $\nu$, then
$L^\infty(Y,{\mathcal Y},\nu)$ can be identified with a subalgebra
of $\M$, and the conditional expectation map becomes its classical
counterpart from probability theory. \fin
\end{example}

Together with $\M$, the orthogonal projection $e_\N$ now generates
in $B(L^2(\tau))$ a larger von Neumann algebra $\langle
\M,e_\N\rangle \supseteq \M$. In general $\langle \M,e_\N\rangle$ is
no longer a finite von Neumann algebra, but it does contain the
dense $\ast$-subalgebra $\A := \lin(\M \cup \{x e_\N y:\ x,y \in
\M\})$ on which we define the \emph{lifted trace} $\bar{\tau}:\A \to
\bbC$ by specifying $\bar{\tau}(xe_\N y) = \tau(xy)$.  By choosing
an orthonormal basis for $L^2(\tau)$ relative to the right action of
$\N$, and consequently realizing $\langle \M,e_\N\rangle$ as an
amplification of $\N$, this linear map is seen to be non-negative
and faithful, and hence defines a semifinite normal faithful
$[0,+\infty]$-valued trace (which we still denote by $\bar \tau$) on
the cone $(\langle\M,e_\N\rangle)^+$ of non-negative (and
self-adjoint) elements of $\langle\M,e_\N\rangle$. This witnesses
that the algebra $\langle\M,e_\N\rangle$ is semifinite (that is, any
positive element of it may be approximated from below by
finite-$\bar{\tau}$ positive elements).  We will not spell out these
standard manipulations here (see, for instance, \cite[Section
1.5]{Pop07}), but we will invoke a notion of orthonormal basis for
right-$\N$-submodules of $L^2(\tau)$ shortly.

\begin{remark} In case $\N\subset \M$ is a finite-index
inclusion of finite II$_1$ factors, then we find that $\langle
\M,e_\N\rangle$ is also a finite II$_1$ factor.  Writing $\M_1$ for
this factor, it follows that the above construction may be repeated
with the inclusion $\M\into \M_1$ in place of $\N\into \M$, and
indeed that it may be iterated to form an infinite tower of II$_1$
factors
\[\N\subset \M\subset \M_1 \subset \M_2 \subset \ldots.\]
This is Jones' \emph{basic construction}; it underlies his famous
work~\cite{Jon83} on the possible values of the index $[\N:\M]$, and
also several more recent developments.  Once again we refer the
reader to~\cite{JonSun97} for a thorough account of its importance,
and numerous further references. However, since the construction of
this whole infinite tower is special to the case of II$_1$ factors,
we will not focus on it further here. \fin
\end{remark}

It is easy to check that the right action of any $n \in \N$ commutes
with any $x e_\N y$, and hence with any member of $\langle
\M,e_\N\rangle$, and in fact it can be shown that $\langle
\M,e_\N\rangle' = \N_{\rm{right}}$ and hence that $\N_{\rm{right}}'
= \langle \M,e_\N\rangle'' = \langle \M,e_\N\rangle$: firstly, if $A
\in B(L^2(\tau))$ commutes with every $b \in \M_{\rm{left}}$ then it
must be the right-action of some $a \in \M$, and now if also
$e_\N(\hat 1 a) = \hat 1 a$ then we must in fact have $a \in \N$
(see Proposition 3.1.2 in~\cite{JonSun97}).  Let us record the
following immediate but important consequence of this for our later
work:

\begin{lem}\label{lem:proj-onto-module}
If $V \leq L^2(\tau)$ is a closed right-$\N$-submodule, then the
orthogonal projection $P_V:L^2(\tau) \to V$ is a member of $\langle
\M,e_\N\rangle$. \qed
\end{lem}

Using $\bar{\tau}$ we can also define an alternative completion of
$\A = \lin\,\M e_\N \M$ for each $p \in [1,\infty)$ by setting
$\|A\|_{p,\bar{\tau}} := \sqrt[p]{\bar{\tau}((A^\ast A)^{p/2})}$ for
$A \in \A$ (where as usual the power $(A^\ast A)^{p/2}$ is defined
using spectral theory for the selfadjoint operator $A^\ast A$, and
the non-negativity of $\bar{\tau}$ is used to show that
$\bar{\tau}((A^\ast A)^{p/2})$ is finite even when $p/2$ is not an
integer). We denote this completion by $L^p(\bar{\tau})$; it is a
Hilbert space when $p=2$. In general elements of $L^p(\bar{\tau})$
do not correspond to elements of $\langle \M,e_\N\rangle$, but they
do give possibly unbounded but closable operators that are weakly
approximable by members of this algebra, which are therefore
affiliated to $\N_{\rm{right}}$.  If $A \in L^p(\bar{\tau})$ is such
an operator that is self-adjoint, then it admits a spectral
decomposition
\[A = \int_\bbR s\,P(\d s)\]
for some spectral measure $P$ on $\bbR$ taking values in the
projections of $\langle \M,e_\N\rangle \cap L^1(\bar{\tau})$, of
possibly unbounded support in $\bbR$, but for which
\[\|A\|^p_{p,\bar{\tau}} = \int_\bbR |s|^p\,\bar{\tau}P(\d s) < \infty.\]

If $V$ is as in Lemma~\ref{lem:proj-onto-module} then we may write
that $P_V$ has \emph{finite lifted trace} if it corresponds to a
member of $\langle \M,e_\N\rangle \cap L^1(\bar{\tau})$.

Now let us introduce some dynamics.  Suppose that $\a$ is a shift
on $\M$ which restricts to a shift on $\N$. Then, as mentioned in
the introduction, $\a$ induces a unitary operator acting on
$L^2(\tau)$, which we shall distinguish from $\a$ by writing it as
$U_\a$; thus for instance
$$U_\a \hat a = U_\a(a\hat 1) = \a(a)\hat 1 = \widehat{\a(a)}$$
for all $a \in \M$.  It is clear that $\overline{\N\hat 1}$ is an
invariant subspace for $U_\a$, so that $U_\a$ commutes with
$e_\N$. Also, conjugation by $U_\a$ agrees with the action $\a$ on
$\M$, thus
$$ U_\a a U_\a^{-1} \xi = \a(a) \xi$$
for all $a \in \M$ and $\xi \in L^2(\tau)$.  Thus, conjugation by
$U_\a$ extends the action of $\a$ to $\langle \M,e_\N\rangle$.

The following special class of one-sided submodules of $L^2(\tau)$
appears here almost exactly as in the commutative setting.

\begin{dfn}[Finite-rank modules]
A left- (respectively, right-) $\N$-submodule $V$ of $L^2(\tau)$
has \emph{finite rank} if there are some $\xi_1$, $\xi_2$, \ldots,
$\xi_r \in V$ such that $V = \overline{\sum_{i=1}^r \N\xi_i}$
(respectively, $V = \overline{\sum_{i=1}^r \xi_i\N}$), and the
numerical value of its \emph{rank} is the least $r \geq 1$ for
which this is possible.
\end{dfn}

\begin{prop}[Relativized Gram-Schmidt procedure]
If $V \leq L^2(\tau)$ is a $U_\a$-invariant right-$\N$-submodule of
finite rank $r$ then there are $\xi_1$, $\xi_2$, \ldots, $\xi_r \in
L^2(\tau)$ such that
\begin{itemize}
\item the subspaces $\overline{\xi_i\N} \leq L^2(\tau)$ are pairwise
orthogonal; and
\item $V = \sum_{i=1}^r\overline{\xi_i\N}$.
\end{itemize}
\end{prop}

\begin{proof} This uses a relativized Gram-Schmidt argument
much as in the commutative setting (see e.g. \cite[Lemma
9.4]{Gla03}).  We proceed by induction on $r$. If $V$ has rank $1$
then the result is immediate from the definition, so let us
suppose that it has rank $r + 1$ for some $r \geq 1$. Then given a
representation
\[V = \overline{\sum_{i=1}^{r+1} \xi_i^\circ \N},\]
we know that any member of $V$ may be approximated in
$\|\cdot\|_{L^2(\tau)}$ by expressions of the form $\xi^\circ_1 n_1 +
\cdots + \xi^\circ_{r+1} n_{r+1}$ for $n_1$, $n_2$, \ldots, $n_{r+1}
\in \N$. This, in turn, may be re-written as
\[(\xi^\perp_1 n_1 + \cdots + \xi^\perp_r n_r) + \big((\xi^\circ_1 - \xi_1^\perp) n_1 + \cdots +
(\xi_r^\circ - \xi^\perp_r) n_r\big) + \xi^\circ_{r+1} n_{r+1}\]
where for each $i \leq r$ we have decomposed $\xi_i^\circ$ into its
component $\xi_i^\perp$ orthogonal to $\overline{\xi_{r+1}\N}$ and
the remainder $\xi^\circ_i - \xi^\perp_i \in
\overline{\xi_{r+1}\N}$.  Since $\overline{\xi_{r+1}\N}$ is a
right-$\N$-submodule, it follows that the second and third inner
sums in the above decomposition both lie in
$\overline{\xi_{r+1}\N}$, and now since
$\overline{\xi_{r+1}\N}^\perp$ is also a right-$\N$-submodule, we
have in fact shown that
\[V = V_1 + \overline{\xi_{r+1}\N}\]
where $V_1 := \overline{\sum_{i=1}^r \xi^\perp_i\N}$ is a rank-$r$
right-$\N$-submodule that is orthogonal to $\overline{\xi_{r+1}\N}$.
Applying the inductive hypothesis to $V_1$ now completes the proof.
\end{proof}

The following definition is also drawn from the commutative world.
This notion has previously been extended to the setting of
non-commutative algebras by Popa in~\cite{Pop07}, who discusses
several other aspects and equivalent conditions in that paper.
(See also \cite{NicStrZsi03}, \cite{Duv09}, \cite{BeyDuvStr07} for
an analysis of the absolute analogue of weak mixing, in which the
subalgebra $\N$ is the trivial algebra $\bbC 1$.)

\begin{dfn}[Relative weak mixing]
If $(\M,\tau,\a)$ is a von Neumann dynamical system and $\N
\subset \M$ is an $\a$-invariant von Neumann subalgebra, then $\a$ is
\emph{weakly mixing relative to $\N$} if for any $a \in \M\cap
\N^\perp$ we have
\[\frac{1}{N}\sum_{n=1}^N\|E_\N(a^\ast \a^n(a))\|^2_\tau \to 0\quad\quad\hbox{as}\ N\to\infty.\]
\end{dfn}

The basic inverse theorem that we need, extending the idea of
Furstenberg and Zimmer to the non-commutative context, is contained
in the following proposition, which essentially re-proves part of
\cite[Lemma 2.10]{Pop07}:

\begin{prop}[Lack of weak mixing implies finite trace submodule]\label{prop:nonrelind-to-invar}
If $\a$ is not weakly mixing relative to $\N$ then there is a
$U_\a$-invariant right-$\N$-submodule $V \leq L^2(\tau)\ominus
\overline{\N\hat 1}$ such that $P_V$ has finite lifted trace.
\end{prop}

\begin{proof} Suppose that $a\in \M\cap \N^\perp$ is such that
\[\frac{1}{N}\sum_{n=1}^N \|E_\N(a^\ast \a^n(a))\|^2_\tau \not\to 0.\]

Define $b := ae_\N a^\ast \in \langle \M,e_\N\rangle$, and now
observe (using the cyclic permutability of $\bar{\tau}$ and the
identity $e_\N m e_\N \equiv E_\N(m)e_\N$) that for any $n \in \bbN$
we have
\begin{multline*}
\bar{\tau}\big(b (U_\a^n b U_\a^{-n})\big) = \bar{\tau}(a e_\N a^\ast
U_\a^n (a e_\N a^\ast) U_\a^{-n}) = \bar{\tau}(a e_\N a^\ast \a^n(a) e_\N
\a^n(a)^\ast)\\
 = \bar{\tau}\big(E_\N(a^\ast\a^n(a)) e_\N \a^n(a)^\ast a\big) =
\|E_\N(a^\ast \a^n(a))\|^2_\tau.
\end{multline*}
Averaging in $n$ it follows that
\[\bar{\tau}\Big(b \frac{1}{N}\sum_{n=1}^N \a^n(b)\Big) \to \langle b,b_1\rangle_{\bar{\tau}} \neq 0\]
where $b_1$ is the limit of the ergodic averages
$\frac{1}{N}\sum_{n=1}^N \a^n(b)$ in the Hilbertian completion
$L^2(\bar{\tau})$, which is therefore invariant under the further
extension of the unitary operator $U_\a$ to this Hilbert space.

This new element $b_1$ need not, in general, correspond to a member
of $\langle \M,e_\N\rangle$ (it is easily seen to be so in the
commutative setting, but for special reasons); however, as a
$\|\cdot\|_{2,\bar{\tau}}$-limit of members of $\langle
\M,e_\N\rangle = \N_{\rm{right}}'$ it can always be identified with
a closed operator on $L^2(\tau)$ that is affiliated with the right
action of the algebra $\N$, and as such it admits a spectral
decomposition
\[b_1 = \int_0^\infty s\,P(\d s)\]
for some resolution of the identity $P$ on $[0,\infty)$ whose
contributing spectral projections lie in $\langle \M,e_\N\rangle$,
and for which
\[\int_0^\infty s^2\bar{\tau}(P(\d s)) = \|b_1\|^2_{2,\bar{\tau}} < \infty.\]
Hence $\bar{\tau}P(I) < \infty$ for any Borel subset $I \subseteq
(0,\infty)$ bounded away from $0$. Now choosing any such subset $I$
for which $P(I) \neq 0$ gives an orthogonal projection $P(I) \in
\langle M,e_\N\rangle$ of finite lifted trace that is
$U_\a$-invariant, commutes with the right-$\N$-action because it
lies in $\langle \M,e_\N\rangle$, and moreover has image orthogonal
to $\overline{\hat 1\N}$ because we initially chose $b$ to lie in the
orthogonal complement of this subspace.
\end{proof}

\begin{remark} The above implication can in fact be reversed,
and these conditions shown to be equivalent to a number of others;
see \cite[Lemma 2.10]{Pop07} for a more complete picture.
\fin
\end{remark}

In the next section we will push the above results a little further
under the additional assumption that the subalgebra $\N$ is central,
leading to the proof of Theorem~\ref{thm:rel-w-m-centre}.

\section{The case of asymptotically abelian systems}

We now specialize to the case of an asymptotically abelian system,
with the crucial additional assumption that the subalgebra $\N$ is
\emph{central}.

\begin{lem}\label{lem:concrete-repn-of-module}
Suppose that $(\M,\tau,\a)$ is a von Neumann dynamical system, $\N \subset \M$ is an $\a$-invariant
central von Neumann subalgebra and $V \leq L^2(\tau)$ is a $U_\a$-invariant
right-$\N$-submodule of finite lifted trace.  Then for any $\eps
> 0$ there is a further $U_\a$-invariant right-$\N$-submodule $V_1 \leq
V$ such that
\begin{itemize}
\item $\bar{\tau}(P_V - P_{V_1}) < \eps$;
\item $V_1$ has finite rank, say $r \geq 1$;
\item there are an orthogonal right-$\N$-basis
$\xi_1$, $\xi_2$, \ldots, $\xi_r$ and a unitary matrix of unitary
operators $U = (u_{ji})_{1 \leq i,j \leq r} \in \cal{U}_{r\times
r}(\N)$ such that
\[U_\a(\xi_i) = \sum_{j=1}^r \xi_j u_{ji}\quad\quad \forall i = 1,2,\ldots,r.\]
\end{itemize}
We refer to $U$ as the \emph{cocycle representing the action of
$U_\a$ on the basis elements $\xi_i$}.
\end{lem}

\begin{proof} We will prove this invoking the picture of the
representation of $\N$ on $L^2(\tau)$ as a direct integral coming
from spectral theory.  By the classical theory of direct integrals
(see, for instance, \cite[Chapter 14]{KadRin97}), we can select
\begin{itemize}
\item a standard Borel probability space $(Y,\nu)$;
\item a Borel partition $Y = \bigcup_{n\geq 1} Y_n \cup Y_\infty$;
\item a collection of Hilbert spaces $\frH_n$ for $n\in
\{1,2,\ldots,\infty\}$ with $\dim(\frH_n) = n$; and
\item a unitary equivalence
\[\Phi:L^2(\tau) \to \frH := \int_Y^\oplus \frH_y\,\nu(\d y),\]
where we define $\frH_y$ to be $\frH_n$ when $y \in Y_n$,
\end{itemize}
such that $\N$ (acting on the right or left, since these agree for a
central subalgebra of $\M$) is identified with the algebra of
functions $L^\infty(\nu)$ acting by pointwise multiplication.
Explicitly, if we denote elements of $\frH$ as measurable sections
$v:Y\to \coprod_{y\in Y}\frH_y$, then $f \in L^\infty(\nu)$ acts on
$\frH$ by
\[M_f(v)(y) := f(y)v(y).\]
Moreover, in order to accommodate $\Phi(\overline{\N\hat 1})$ we
select a measurable section $v_0\in \frH$ with $\|v_0(y)\|_{\frH_y}
\equiv 1$, and now $\N\hat 1$ is identified with
\[\{y\mapsto f(y)v_0(y):\ f\in L^\infty(\mu)\},\]
so that the orthogonal projection $\Phi e_\N\Phi^{-1}$ acts by
\[\Phi e_\N\Phi^{-1}(v)(y) := \langle v(y),v_0(y)\rangle_{\frH_y}\cdot v_0(y).\]

The larger algebra $\M_{\rm{right}}$ is identified under $\Phi$ with
a direct integral
\[\int_Y^{\oplus} \M_y\,\nu(\d y),\]
so that elements of $\Phi(\M)$ are expressed as measurable sections
$T:Y\to \coprod_{y\in Y}B(\frH_y)$ acting by
\[Tv(y) := T(y)(v(y))\]
and such that $T(y) \in \M_y$ $\nu$-almost surely, where $(\M_y)_{y
\in Y}$ is a measurable field of finite von Neumann subalgebras of
$B(\frH_y)$ for each of which the state \[\M_y\to \bbC:T \mapsto
\langle v_0(y),T(v_0(y))\rangle_{\frH_y}\] is a faithful finite
trace; overall we have
\[\tau(a) = \langle \hat 1,a\hat 1\rangle = \int_Y \langle v_0(y),\Phi(a)(y)(v_0(y))\rangle_{\frH_y}\,\nu(\d y)\]
for $a\in \M$, and so in particular if $n\in \N$ then $\Phi(n) \in
L^\infty(\mu)$ and $\tau(n) = \int \Phi(n)\,\d\nu$.

Given these data, for $a,b\in \M$ we can compute that
\[\Phi(ae_\N b)\Phi^{-1}v(y) = \langle\Phi(b)(y)(v(y)),v_0(y)\rangle\cdot \Phi(a)(y)(v_0(y))\]
and
\begin{multline*}
\bar{\tau}(a e_\N b) = \tau(ab) = \int_Y \langle
v_0(y),\Phi(ab)(y)(v_0(y))\rangle_{\frH_y}\,\nu(\d y)\\
= \int_Y \langle
\Phi(a^\ast)(y)(v_0(y)),\Phi(b)(y)(v_0(y))\rangle_{\frH_y}\,\nu(\d
y) = \int_Y \rm{tr}(\Phi(a e_\N b)\Phi^{-1}|_{\frH_y})\,\nu(\d y).
\end{multline*}

In this representation an $\N$-submodule $V \leq L^2(\tau)$
corresponds to a subspace $\Phi(V) \leq \frH$ of the form
$\int_Y^\oplus V_y\,\nu(\d y)$ for some measurable subfield of
Hilbert spaces $V_y \leq \frH_y$, and the above calculation now
shows that
\[\bar{\tau}(P_V) = \int_Y \rm{dim}(V_y)\,\nu(\d y),\]
so $P_V$ has finite lifted trace if and only if the function
$y\mapsto \dim(V_y)$ is $\nu$-integrable.

We can enhance this picture further by noting that since $\a$
preserves $\N$ it must correspond to some $\nu$-preserving
transformation $S\curvearrowright Y$, and that since it also
preserves $\M$ and extends to a unitary operator on $L^2(\tau)$ it
must also preserve each of the cells $Y_n$.  Similarly, since $V$ is
$U_\a$-invariant, the transformation $S$ must preserve the function
$y\mapsto \deg(V_y)$. It follows that the unitary operator $\Phi
U_\a \Phi^{-1}$ on $L^2(\tau)$ is actually given by a measurable
section of unitary operators $\Psi:Y\to \coprod_{y \in
Y}\cal{U}(\frH_y)$ such that
\[\Phi U_\a \Phi^{-1}v(y) = \Psi(y)(v(S^{-1}y)).\]

Now, since $y \mapsto \deg(V_y)$ is $\nu$-integrable, for
sufficiently large $r \geq 1$ we know that
\[\int_{\{y\in Y:\ \deg(V_y) > r\}}\deg(V_y)\,\nu(\d y) < \eps.\]
Define
\[W := \int_{\{y\in Y:\ \deg(V_y) \leq r\}}^\oplus V_y\,\nu(\d y) \oplus
\int_{\{y\in Y:\ \deg(V_y) > r\}}^\oplus\{0\}\,\nu(\d y)\]
and $V_1 := \Phi^{-1}(W)$. Clearly $V_1$ is still a
right-$\N$-submodule that is $U_\a$-invariant, and it clearly also
has rank at most $r$ (since it suffices to prove this for $W$, for
which it follows by a relativized Gram-Schmidt construction of a
fibrewise-orthonormal basis exactly as in the setting of commutative
ergodic theory; see for instance \cite[Lemma 9.4]{Gla03}). Also, we have
\[\bar{\tau}(P_V - P_{V_1}) = \int_{\{y\in Y:\ \deg(V_y) > r\}}\deg(V_y)\,\nu(\d y) <
\eps.\]

Finally, the selection of unitaries $\Psi$ must preserve the field
of subspaces $V_y$ above the $S$-invariant set $\{y\in Y:\ \deg(V_y)
= s\}$ for each $s\leq r$.  Choosing an abstract $d$-dimensional
Euclidean space $W_d$ for each $d \leq r$ and adjusting each fibre
of $W$ by a unitary in order to identify each $V_y$ for which
$\dim(V_y)\leq r$ with $W_{\dim(V_y)}$, we obtain a new
representation of $V_1$ as a right-$\N$-submodule using these fibres
$W_d$, so that the action of $U_\a$ is now described by a measurable
family of unitaries $\Psi'(y) \in \cal{U}(W_{\dim(V_y)})$. Picking
an orthonormal basis for each $W_d$, writing these unitary operators
as unitary matrices in terms of these bases, noting that their
individual entries are now identified with elements of
$L^\infty(\mu) = \Phi(\N)$ and carrying everything back to
$L^2(\tau)$ using $\Phi^{-1}$ gives the desired expression for
$U_\a$.
\end{proof}

\begin{remark} Frustratingly, both the fact that a
$U_\a$-invariant $V$ of finite lifted trace may be approximated by a
$U_\a$-invariant $V_1$ of finite rank, and the fact that given such
a module of finite rank the action of $U_\a$ on it may be described
by a unitary element in $\cal{U}(M_{r\times r}(\N))$, seem to be
difficult to prove without the assumption that $\N$ is central and
the resulting representation of the action of $\N$ on $L^2(\mu)$ as
the multiplication action of some $L^\infty(\nu)$ on a measurable
field of Hilbert spaces.  It would be interesting to settle this
issue more generally:

\begin{ques}
Do these conclusions hold for a finite-lifted-trace invariant
sub\-mo\-dule corresponding to an arbitrary inclusion of finite
von Neumann algebras with a trace-preserving automorphism? \fin
\end{ques}
\end{remark}

Before moving on let us quickly note an important difference from
the setting of abelian von Neumann algebras.

\begin{example} If $\M$ is abelian, then from commutative
ergodic theory it is well-known that all the intermediate
$U_\a$-invariant submodules $V \leq L^2(\tau)$ that have finite-rank
over $\N$ together generate an intermediate sub\emph{algebra}
between $\N$ and $\M$, and that this then corresponds to an
intermediate measure-preserving system.  We will see shortly that an
analogous conclusion can sometimes be recovered in the
asymptotically abelian setting, but it is certainly not true for
general finite-rank submodules, even when the smaller algebra $\N$
is abelian.

Consider, for example, the inclusion $i:L\bbZ \cong
L^\infty(m_\bbT)\into L\bbF_2$ corresponding to the embedding of
$\bbZ$ as the cyclic subgroup $a^\bbZ$ of the free group $\bbF_2 =
\langle a,b\rangle$.  Here $LG$ is the group von Neumann algebra of
$G$, defined in Section \ref{4fail-sec}.  In this case we can
identify $L^2(\tau)$ as $\ell^2(\bbF_2)$ and $L^2(\tau|_\N)$ as the
subspace spanned by $\{\xi_{a^n}\}_{n\in \bbZ}$.  Now define $\a\in
\rm{Aut}\,L\bbF_2$ simply by lifting the group automorphism of
$\bbF_2$ that fixes $a$ and maps $b\mapsto ba$.  Now the subspace $V
:= \overline{\lin}\{\xi_{ba^n}:\ n\in\bbZ\}\leq \ell^2(\bbF_2)$ is a
$U_\a$-invariant right $\N$-module of rank one which is orthogonal
to $L^2(\tau|_\N)$.  On the other hand, although $\xi_b \in
\widehat{\M}\cap V$, we have $\a^m(\xi_b^2) = \a^m(\xi_{b^2}) =
\xi_{ba^mba^m}$ for $m\in\bbZ$, and it is easy to see that these
elements of $\M$ do not remain within any finite-rank
right-$\N$-submodule.

It is true that if $L^2(\tau)\ominus L^2(\tau|_\N)$ contains a
finite-rank right-$\N$-submodule $V$, then it also contains a
finite-rank left-$\N$-module in the form of $J(V)$, where $J$ is
the \emph{modular automorphism} on $V$, defined by extending the
conjugation map $a \mapsto a^*$ on $\M \equiv \hat \M$ by density.
The point is that it can happen that $J(V) \perp V$, and that all
elements of $J(V)$ are weakly mixed by $U_\a$: it is the
right-module $V$, and no other, that serves as the obstruction to
overall relative weak mixing coming from
Theorem~\ref{thm:rel-w-m-red}. \fin
\end{example}

We now introduce a useful technical concept.

\begin{dfn}[Central vectors]
A vector $\xi \in L^2(\tau)$ is \emph{central} if $m\xi = \xi m$
for all $m\in \M$.
\end{dfn}

\begin{lem}[No non-obvious central
vectors]\label{lem:no-nonobvious-centrals} The closure
$\overline{\Z(\M)\hat 1} = \overline{\hat 1\Z(\M)}$ is equal to the
set of all central vectors in $L^2(\tau)$.
\end{lem}

\begin{proof} Suppose that $\xi \in L^2(\tau)$ is central.
 Define $a_\xi:\M\hat 1 \to L^2(\tau)$ by $a_\xi(m\hat 1) := \xi m$.
This is a densely-defined linear operator on $L^2(\tau)$, and it is
closable because if $m_n\hat 1 = \hat 1m_n\to 0$ in $\|\cdot\|_{L^2(\tau)}$
for some sequence $(m_n)_{n \geq 1}$ in $\M$ and also $\xi m_n \to
\xi'$ in $\|\cdot\|_{L^2(\tau)}$, then we have
\[\langle m'\hat 1,\xi'\rangle = \lim_{n\to\infty}\langle m'\hat 1,\xi m_n\rangle
= \lim_{n\to\infty}\langle \hat 1 m_n^\ast,(m')^\ast\xi\rangle = 0\]
for every $m' \in \M$, and so in fact we must have $\xi' = 0$. Also,
we clearly have \[a_\xi(m\hat 1) = a_\xi(\hat 1 m) = \xi m = m\xi =
(a_\xi(\hat 1))m = m(a_\xi(\hat 1))\] for every $m \in \M$, so
$a_\xi$ is affiliated with both the right- and left-actions of $\M$
on $L^2(\tau)$. The same therefore holds for $a_\xi + a^\ast_\xi$
and $\rm{i}(a_\xi - a^\ast_\xi)$, and now these are self-adjoint and
so each of them may be expressed as an unbounded spectral integral
all of whose contributing spectral projections must lie in
$\M_{\rm{left}}' \cap \M_{\rm{right}}' = \Z(\M)$.  Therefore,
approximating $a_\xi = \frac{1}{2}(a_\xi + a^\ast_\xi) +
\frac{1}{2}(a_\xi - a^\ast_\xi)$ by a sum of two large but bounded
integrals with respect to the respective resolutions of the
identity, we obtain a sequence of elements $a_n \in \Z(\M)$ such
that $a_n\to a_\xi$ pointwise on $\rm{dom}(\rm{clos}(a_\xi))
\supseteq \M\hat 1$, and hence such that $a_n\hat 1 \to \xi$ in
$\|\cdot\|_{L^2(\tau)}$. Hence $\xi \in \overline{\Z(\M)\hat 1}$, as
required.
\end{proof}

\begin{prop}\label{prop:f.r.central} If $(\M,\tau,\a)$ is an asymptotically abelian von Neumann dynamical system,
$\N$ is a shift-invariant central von Neumann subalgebra, and $V \leq
L^2(\tau)$ is an $\a$-invariant right-$\N$-submodule of $\M$
having finite lifted trace then all elements of $V$ are central
vectors.
\end{prop}

\begin{proof}
Clearly it will suffice to prove this for all finite-rank
approximants $V_1$ to $V$ as given by
Lemma~\ref{lem:concrete-repn-of-module}. Thus we may assume that $V$
actually has finite rank. Let $\xi_1$, $\xi_2$, \ldots, $\xi_r$ and
$U = (u_{ji})_{1\leq i,j \leq r}\in \M_{r\times r}(\N)$ be as given
by the third part of that lemma.

Since $\a$ is asymptotically abelian, we have for any $a\hat 1 \in
\M\hat 1$ and $b \in \M$ that
\[\frac{1}{N}\sum_{n=1}^N\|bU_\a^n(a\hat 1) - U_\a^n(a\hat 1)b\|_{L^2(\tau)}
= \frac{1}{N}\sum_{n=1}^N\|b\a^n(a) - \a^n(a)b\|_{L^2(\tau)} \to 0.\]

Approximating an arbitrary $\xi \in L^2(\tau)$ by elements of
$\M\hat 1$, it follows that for each fixed $b \in \M$ and $\xi \in
L^2(\tau)$ we have
\[\lim_{N\to\infty}\frac{1}{N}\sum_{n=1}^N\|bU_\a^n(\xi) -
U_\a^n(\xi)b\|_{L^2(\tau)} = 0.\]

On the other hand, we know that
\[U_\a(\xi_i) = \sum_{j=1}^r \xi_ju_{ji}\quad\quad\forall i=1,2,\ldots,r,\]
and so, writing $U^n = (u^{(n)}_{ji})_{1\leq i,j\leq r}$, we have
\begin{multline*}
U^{-n}_\a(\xi_i) = \sum_{j=1}^r \xi_ju^{(-n)}_{ji}\quad\quad
\Rightarrow\quad\quad \xi_i = \sum_{j=1}^r
U_\a^n(\xi_j)\a^n(u^{(-n)}_{ji})\quad\quad\forall i=1,2,\ldots,r.
\end{multline*}

Clearly each $u^{(-n)}_{ji}$ is still a unitary, and so from this,
averaging in $n$ and the centrality of $\N$ we obtain
\begin{eqnarray*}
\|b\xi_i - \xi_i b\|_{L^2(\tau)} &=&
\Big\|\frac{1}{N}\sum_{n=1}^N\Big(\sum_{j=1}^r
bU_\a^n(\xi_j)\a^n(u^{(-n)}_{ji}) - \sum_{j=1}^r
U_\a^n(\xi_j)\a^n(u^{(-n)}_{ji})b\Big)\Big\|_{L^2(\tau)}\\
&=& \Big\|\frac{1}{N}\sum_{n=1}^N\sum_{j=1}^r \big(bU_\a^n(\xi_j)
-
U_\a^n(\xi_j)b\big)\a^n(u^{(-n)}_{ji})\Big\|_{L^2(\tau)} \\
&\leq& \sum_{j=1}^r\frac{1}{N}\sum_{n=1}^N \|bU_\a^n(\xi_j) -
U_\a^n(\xi_j)b\|_{L^2(\tau)},
\end{eqnarray*}
and now since each of the summands in $j$ tends to $0$ as $N\to
\infty$, it follows that we must in fact have $b\xi_i = \xi_i b$ for
every $i \leq r$, and hence (taking $\N$-linear combinations, which
have central coefficients, and then a completion) that all vectors
in $V$ are central, as required.
\end{proof}

Let us note explicitly the following simple corollary of the above
result.

\begin{cor}\label{cor:invts-central}
If $(\M,\tau,\a)$ is an asymptotically abelian von Neumann
dynamical system, then the subalgebra $\M^\a := \{ a \in \M: \a(a)
= a \}$ of individually $\a$-invariant elements is central.
\end{cor}

\begin{proof} Of course, if $\a(a) = a$ then $\lin\{\hat 1
a\}$ is a rank-one $\a$-invariant submodule of $L^2(\tau)$ for the
trivial central subalgebra $\N := \bbC \hat 1$, and the claim follows from Proposition \ref{prop:f.r.central}.  This claim can also be easily verified directly from the definition of asymptotic abelianness.
\end{proof}

Finally we can use the above results to prove
Theorem~\ref{thm:rel-w-m-centre}.

\begin{proof} (Proof of Theorem~\ref{thm:rel-w-m-centre}) Suppose, for
the sake of contradiction, that $\a$ were not weakly mixing relative
to $\Z(\M)\subset \M$.  Then
Proposition~\ref{prop:nonrelind-to-invar} gives a nontrivial
right-$\Z(\M)$-submodule $V \leq L^2(\tau) \ominus
\overline{\Z(\M)\hat 1}$ of finite lifted trace, and now
Proposition~\ref{prop:f.r.central} tells us that $V$ must consist of
central vectors. However, Lemma~\ref{lem:no-nonobvious-centrals} now
gives $V \leq \overline{\Z(\M)\hat 1}$, implying a contradiction with
our assumption that $V\perp \overline{\Z(\M)\hat 1}$.
\end{proof}

Note that for the results in this section it suffices to assume that
for every $a\in\M$ there exists a sequence $\{n_j\}$ such that
$\lim_{j\to\infty}\|[\a^{n_j}(a),b]\|_{L^2(\tau)}=0$ for every
$b\in\M$.   We do not know whether this condition is strictly weaker than asymptotically abelianness.

\begin{remark} A variant of Theorem~\ref{thm:rel-w-m-centre} can also be deduced from the results in \cite{NicStrZsi03} (and more specifically, Theorem 4.2 and Proposition 5.5 of that paper); we thank the anonymous referee for pointing out this fact.  More specifically, the result is that if $\alpha$ is an automorphism of a finite von Neumann algebra $\M$ that leaves invariant a faithful normal trace $\tau$, and $E_\tau$ is the conditional expectation to the factor
$$
\M_r := \overline\lin^{\mathrm{wot}}\{a\in \M:\ \a (a)=\lambda a \text{ for some } \lambda\in\bbT\},$$
then  for any $a,b \in \M$ one has
$$ \lim_{N \to \infty} \frac{1}{N} \sum_{n=1}^N |\langle E_\tau(a^* \alpha^n(a) )- E_\tau(a)^* \alpha^n(E_\tau(a)), b \rangle_{L^2(\tau)}| = 0;$$
in particular, for $N$ going to infinity along a density one set of integers, the expression $E_\tau(a^* \alpha^n(a) )- E_\tau(a)^* \alpha^n(E_\tau(a))$ converges to zero in the weak operator topology.  This property is weaker than the relative weak mixing property with respect to this factor (which one does not expect to hold in general, even in the abelian case), but on the other hand does not require any hypothesis of asymptotic abelianness.
\end{remark}

\section{Triple averages for non-asymptotically-abelian systems}\label{triple}

The use to which we put relative weak mixing in the preceding
section is very special to asymptotically abelian systems: in
general there seems to be no way to track the error term resulting
from the re-arrangement at the heart of the proof of
Theorem~\ref{thm:rel-w-m-red} without this assumption. However, in
the special case of triple averages this problem does simplify
somewhat, provided we assume instead that our system
$(\M,\tau,\alpha)$ is \emph{ergodic}, so that $\M^\alpha = \bbC 1$.
In this case we will be able to obtain convergence weakly and in
norm, as well as recurrence on a dense set (Theorem
\ref{thm:gen-triple-aves}).

This assumption is not so innocuous as might be expected from its
analog in the world of commutative ergodic theory. In that setting
it is possible quite generally to decompose a system (that is,
more precisely, to decompose its invariant measure) into ergodic
components, and then many assertions about the whole system,
including multiple recurrence and the convergence of multiple
averages, follow if they can be proved for each ergodic component
separately.  However, in the setting of a general von Neumann
dynamical system, this decomposition is available only if $\M^\a$
is central in $\M$; otherwise the automorphism $\a$ can exhibit
genuinely new phenomena precisely in virtue of having the
nontrivial fixed subalgebra $\M^\a$ to `move around'.  This was
already seen in the failure of recurrence on a dense set when the
ergodicity hypothesis is dropped (Theorem \ref{3fail}).

The key for convergence of triple averages is the following
decomposition similar to the commutative case, first established (in a slightly more general setting) in \cite{NicStrZsi03} (and more specifically, from Theorem 4.2 and Proposition 5.5 in that paper); for the convenience of the reader we give a short proof of that decomposition here. Note that the result does not require ergodicity of the system.  We remark that a closely related decomposition was also used in \cite{fidaleo28}.

\begin{prop}[Decomposition of von Neumann dynamical systems]\label{prop:decomposition}\cite{NicStrZsi03}
Let $(\M,\tau,\a)$ be a von Neumann dynamical system.  Then one
has the orthogonal decomposition $\M=\M_r\oplus \M_s$, where
\begin{eqnarray*}
\M_r&:=&\overline\lin^{\mathrm{wot}}\{a\in \M:\ \a (a)=\lambda a \text{ for some } \lambda\in\bbT\}
 \text{ and}
\\
\M_s&:=&\left\{a\in \M:\  \lim_{N\to\infty} \frac{1}{N}
\sum_{n=1}^N |\tau(b\, \a^n (a))|=0 \quad \text{for every } b\in
\M\right\},
\end{eqnarray*}
i.e., $\M_r$ is the von Neumann subalgebra spanned by the eigenvectors of
$\alpha$ and $\M_s$ is the subspace of the elements of $\M$ that
are weakly mixed by $\alpha$. The corresponding projection onto
$\M_r$ is the conditional expectation of $\M$ onto $\M_r$ and in
particular preserves positivity.
\end{prop}
\begin{proof} Since the continuation $U_\a$ of $\a$ to
$L^2(\tau)$ is a unitary operator, the Jacobs--Glicksberg--de
Leeuw decomposition holds for $U_\a$ (see e.g. \cite[Section
2.4]{Kre85}), i.e., $L^2(\tau)=L^2_r(\tau)\oplus L^2_s(\tau)$,
where the \emph{reversible part} $L^2_r(\tau)$ is defined as
$$ L^2_r(\tau)=\overline\lin \{x: U_\a
(x)=\lambda x \text{ for some } \lambda\in\bbT\}$$
and the \emph{stable part} $L^2_s(\tau)$ is defined as the space of all $x\in L^2(\tau)$ such that
$$
  \lim_{N\to\infty} \frac{1}{N} \sum_{n=1}^N |\langle U_\a^n (x), y\rangle|=0 \quad \text{for every } y\in L^2(\tau).
$$
Moreover, this decomposition is orthogonal since $U_\a$ is
unitary. Note that we do not need here the Jacobs--Glicksberg--de
Leeuw decomposition in full generality but only its version for
unitary operators, which can be also proved via the spectral
theorem.

By a result of St\o rmer~\cite{stormer:1974}, the eigenvectors of
$U_\a$ belong to $\M$. We thus have $\M_r=\M\cap L^2_r(\tau)$ and
$\M_s=\M\cap L^2_s(\tau)$. The fact that the weak operator closure
and the closure in the $L^2(\tau)$-topology coincide for
self-adjoint subalgebras implies the second formula for $\M_r$ and
thus $\M_r$ is a von Neumann subalgebra of $\M$. The conditional
expectation now maps $\M$ onto $\M_r$ assuring the orthogonal
decomposition $\M=\M_r\oplus \M_s$.
\end{proof}

In the remainder of this section we assume our system is ergodic.

\begin{prop}[Convergence of triple averages]
Let $(\M,\tau,\a)$ be an ergodic von Neumann
dynamical system. Then the averages
\begin{equation}\label{eq:convergence}
  \frac{1}{N} \sum_{n=1}^N \a^n (a)\, \a^{2n} (b)
\end{equation}
converge in $\|\cdot\|_{L^2(\tau)}$ as $N\to \infty$ for every $a,b\in
\M$.
\end{prop}
\begin{proof} By the above proposition, it suffices to
assume that $a$ and $b$ each belong to $\M_r$ or $\M_s$.
Suppose first that $a\in \M_r$, and fix $b$. The operators $S_N$
given by
$$
S_N x=\frac{1}{N} \sum_{n=1}^N \a^n (x)\, \a^{2n} (b)
$$
are linear and bounded on $\M$ for the norm $\|\cdot\|_{L^2(\tau)}$, so
we may assume that $\a (a)=\lambda a$ for some $\lambda\in\bbT$.
Then $S_Na=\frac{1}{N+1} \sum_{n=0}^N a (\lambda\a^2)^n (b)$ which
converges in $L^2(\tau)$ by the mean ergodic theorem.

Suppose now that $a\in \M_s$. We show that the desired limit is
zero. Consider $u_n:=\a^n (a)\, \a^{2n} (b)\, \hat 1$ and observe
that
\begin{eqnarray*}
 \langle u_n, u_{n+j} \rangle &=& \tau (\a^{2n} (b^*)\, \a^{n} (a^*)\, \a^{n+j} (a)\, \a^{2n+2j}
 (b))\\
  &=& \tau(\a^{n} (b^*)\,a^*\, \a^{j} (a)\, \a^{n+2j} (b))
  = \tau (a^*\,\a^{j} (a)\, \a^{n} (\a^{2j}(b)\, b^*)).
\end{eqnarray*}
The ergodicity of the system implies
\begin{eqnarray*}
 \gamma_j&:=&\lim_{N\to \infty} \left|\frac{1}{N} \sum_{n=1}^N  \langle u_n, u_{n+j} \rangle \right|\\
  &=&  \left|\tau \left(a^*\,\a^{j} (a) \lim_{N\to \infty}
  \frac{1}{N} \sum_{n=1}^N
      \a^{n} (\a^{2j}(b)\, b^*)\right)\right|
  = |\tau (a^*\, \a^{j} (a))|\cdot |\tau(\a^{2j}(b)\, b^*)|.
\end{eqnarray*}
Since $a\in \M_s$ and $\tau(\a^{2j}(b)\, b^*)$  are bounded in
$j$, we have
$$
 \lim_{N\to \infty} \frac{1}{N} \sum_{j=1}^N  \gamma_j =0,
$$
and therefore by the classical van der Corput lemma for Hilbert
spaces (see e.g. \cite{Fur77} or \cite{Ber87}), we have
$\lim_{N\to \infty} \frac{1}{N} \sum_{n=1}^N u_n=0$.
\end{proof}

\begin{remarks}
\begin{enumerate}
\item For compact non-ergodic systems the averages
(\ref{eq:convergence}) converge as well, since $\M=\M_r$ in this
case; this was also observed in \cite{BeyDuvStr07}.

\item As in the commutative case we see that the Kronecker subalgebra
$\M_r$ is characteristic for (\ref{eq:convergence}), i.e., the
limit of the averages in (\ref{eq:convergence}) does not change if
replacing $a$ by $E_{\M_r}a$ and $b$ by $E_{\M_r}b$. \fin
\end{enumerate}
\end{remarks}

As was shown in Corollary \ref{negns}, one cannot expect the limit
$$
\lim_{N\to \infty} \frac{1}{N} \sum_{n=1}^N \tau(a \a^n (a)
\a^{2n} (a))
$$
to be positive for every positive $a$. However, a modification
extending \cite[Theorem 5.13]{BeyDuvStr07} is still true.

\begin{prop}
For an ergodic von Neumann system $(\M,\tau,\a)$, one has
$$
 \liminf_{N\to \infty} \frac{1}{N} \sum_{n=1}^N (\Re \tau(a\, \a^n (a) \, \a^{2n}
 (a)))_+ >0
$$
for every $0<a\in\M$.  In particular, one has recurrence on a
dense set.
\end{prop}
\begin{proof}
Decompose $a=b+c$ with $b\in \M_r$ and $c\in
\M_s$ as in Proposition \ref{prop:decomposition}, with $b>0$ by
Lemma \ref{lem:condexp}.
We first show that there exists a compact abelian group $G$, an
open set $U\subset G$ and $g\in G$ such that for the $1$-step Bohr
set $K_U:=\{n\in \bbN: g^n \in U\}$ one has
\begin{equation}\label{eq:positivity-compact}
\Re\tau(b\, \a^n (b)\, \a^{2n} (b)) > \frac{\tau (b^3)}{2}>0 \quad
\text{for every } n\in K_U.
\end{equation}

Take $\eps:=\frac{\tau (b^3)}{18\|b\|^2}$. Since $b\in \M_r$, we
find $k\in \bbN$, $\lambda_1,\ldots \lambda_k\in\bbT$ and
$b_1,\ldots, b_k\in \M\setminus\{0\}$ such that $\a (b_j)=\lambda_j
b_j$ for every $j=1,\ldots, k$ and $\|b-(b_1+\ldots
+b_k)\|_{L^2(\tau)}<\eps$. Set now $G:=\bbT^k$, $g:=(\lambda_1,
\ldots, \lambda_k)$ and $U:=U_{\eps/(k\max\|b_j\|)}(1)\subset
\bbT^k$. Observe that for every $n$ such that $g^n\in U$, we have
$|\lambda_j^n-1|<\varepsilon/(k\max\|b_j\|)$ for every $j=1,\ldots,
k$ and therefore
\begin{eqnarray*}
\|\a^n (b) - b\|_{L^2(\tau)}&\leq& \|a^n (b_1+\ldots+b_k) -
(b_1+\ldots+b_k)\|_{L^2(\tau)} \\
&&\quad\quad +2\|b_1+\ldots+b_k - b\|_{L^2(\tau)}\\ &\leq&
\max\|b_j\|_{L^2(\tau)} (|\lambda_1^n -1|+\ldots + |\lambda_k^n
-1|) + 2\eps \\
&<& \max\|b_j\| \frac{k\eps}{k\max\|b_j\|} + 2\eps = 3 \eps.
\end{eqnarray*}
So we have by the Cauchy-Schwarz inequality
\begin{eqnarray*}
|\tau(b\, \a^n (b)\, \a^{2n} (b)) - \tau(b^3)| &\leq& |\tau(b\, \a^n (b)\, (\a^{2n}(b)-b))|+
|\tau(b (\a^n (b)-b) b )|\\
&\leq& \|b\|^2 (\|\a^{2n}(b)-b\|_{L^2(\tau)}+\|\a^{n}(b)-b\|_{L^2(\tau)})\\
&\leq& 3  \|b\|^2 \|\a^{n}(b)-b\|_{L^2(\tau)} \, <\, 9 \|b\|^2 \eps=
\frac{\tau (b^3)}{2},
\end{eqnarray*}
and (\ref{eq:positivity-compact}) is proved.

Take now $V:= U_{\eps/(2k\max\|b_j\|)}(1)\subset U$ and a
continuous function $f:G\to [0,1]$ satisfying $\one_V\leq f \leq
\one_U$. Then by (\ref{eq:positivity-compact}) $\Re\tau(b\, \a^n
(b)\, \a^{2n} (b))$ is positive whenever $f(g^n)\neq 0$ and
therefore
$$
  \liminf_{N\to \infty} \frac{1}{N} \sum_{n=1}^N f(g^n) \Re\tau(b\, \a^n (b)\, \a^{2n} (b)) \geq
  \liminf_{N\to \infty} \frac{1}{N} \sum_{n=1}^N \one_V(g^n) \Re\tau(b\, \a^n (b)\, \a^{2n} (b)).
$$
Since the set $K_{V}:=\{n\in \bbN: g^n \in V\}\subset K_U$ is
syndetic (i.e. has bounded gaps) in $\bbN$, this implies by (\ref{eq:positivity-compact})
\begin{equation}\label{eq:positivity-compact-2}
  \liminf_{N\to \infty} \frac{1}{N} \sum_{n=1}^N f(g^n) \Re\tau(b\, \a^n (b)\, \a^{2n} (b))>0.
\end{equation}

Next, we show that
\begin{equation}\label{eq:positivity-compact-zero}
\|\cdot\|_{L^2(\tau)}-\lim_{N\to \infty} \frac{1}{N} \sum_{n=1}^N f(g^n)
\a^n (b)\,\a^{2n}(c)=0.
\end{equation}
To do this, we first consider a character $\g\in \hat{G}$ and
define
$$
  u_n:=\g (g^n)\a^n (b)\, \a^{2n} (c)\, \hat 1.
$$
We have
\begin{eqnarray*}
  \langle u_n, u_{n+j} \rangle &=& \ol{\g (g^n)} \g (g^{n+j}) \g (\a^{2n} (c^*)\, \a^{n} (b^*)\, \a^{n+j}
  (b)\, \a^{2n+2j} (c)) \\
  &=& \g (g^j) \tau(\a^{n} (c^*)\, b^*\,\a^{j} (b)\, \a^{n+2j} (c))
  = \g (g^j) \tau (b^*\, \a^j (b)\, \a^{n} (\a^{2j}(c)\,
  c^*)).
\end{eqnarray*}
By ergodicity of $\a$,
\begin{eqnarray*}
 \gamma_j&:=&\lim_{N\to \infty} \left|\frac{1}{N} \sum_{n=1}^N  \langle u_n, u_{n+j} \rangle \right|
  = \left|\g (g^j)\tau \left(b^*\, \a^j (b) \lim_{N\to \infty} \frac{1}{N} \sum_{n=1}^N  \a^{n} (\a^{2j}(c)\, c^*)
      \right)\right|\\
  &=& |\tau (b^*\, \a^{j} (b))|\cdot |\tau(\a^{2j}(c)\, c^*)|,
\end{eqnarray*}
and the assumption $c\in \M_s$ implies $\lim_{N\to \infty}
\frac{1}{N} \sum_{j=1}^N \gamma_j=0$. By the van der Corput
estimate we thus have
$$
  \lim_{N\to \infty} \frac{1}{N} \sum_{n=1}^N u_n = \lim_{N\to \infty} \frac{1}{N}
  \sum_{n=1}^N \gamma (g^n)\a^n (b)\, \a^{2n} (c) \hat 1 =0.
$$
Since the characters form a total set in $C(G)$ and the operators
$$
S_Nf:=\frac{1}{N} \sum_{n=1}^N f (g^n)\a^n (b)\, \a^{2n} (c)
$$
are uniformly bounded on $C(G)$,
(\ref{eq:positivity-compact-zero}) is proved. Analogously one also
has
\begin{eqnarray*}
  \|\cdot\|_{L^2(\tau)}-\lim_{N\to \infty} \frac{1}{N} \sum_{n=1}^N f(g^n) \a^n (c)\, \a^{2n}
 (b)
  = \|\cdot\|_{L^2(\tau)}-\lim_{N\to \infty} \frac{1}{N} \sum_{n=1}^N f(g^n) \a^n (c)\,
  \a^{2n}(c) =0.
\end{eqnarray*}
The Cauchy-Schwarz inequality implies now that
\begin{eqnarray*}
\limsup_{N\to \infty} \left|\frac{1}{N} \sum_{n=1}^N f(g^n)
\tau(c\, \a^n (b)\, \a^{2n}
 (c))\right|= \limsup_{N\to \infty}\left|\tau\left(c\, \frac{1}{N} \sum_{n=1}^N f(g^n)\a^n (b)\, \a^{2n}
 (c)\right)\right|\\
 \hfill \leq \|c\|_{L^2(\tau)} \limsup_{N\to \infty} \left\|\frac{1}{N} \sum_{n=1}^N
f(g^n)\a^n (b)\, \a^{2n}
 (c)\right\|_{L^2(\tau)}=0,
\end{eqnarray*}
and analogously for the Ces\`aro sums of $f(g^n) \tau(c\, \a^n
(c)\, \a^{2n} (b))$, $f(g^n) \tau(c\, \a^n (c)\, \a^{2n} (c))$ and
$f(g^n) \tau(b\, \a^n (c)\, \a^{2n} (c))$ while
$$
\tau(c\, \a^n (b)\, \a^{2n} (b))= \tau(b\, \a^n (b)\, \a^{2n} (c))
=\tau(b\, \a^n (c)\, \a^{2n} (b))=0
$$
follows from the orthogonality of $\M_r$ and $\M_s$ and the fact
that $\M_r$ is an $\a$-invariant self-adjoint subalgebra of $\M$.

Combining this with (\ref{eq:positivity-compact-2}), we obtain by
the linearity of $\tau$
\begin{eqnarray*}
\liminf_{N\to \infty} \frac{1}{N} \sum_{n=1}^N (\Re \tau(a\, \a^n (a)
\, \a^{2n} (a)))_+
&\geq& \liminf_{N\to \infty} \frac{1}{N} \sum_{n=1}^N f(g^n)(\Re\tau(a\, \a^n (a) \, \a^{2n} (a)))_+\\
&=&  \liminf_{N\to \infty} \frac{1}{N} \sum_{n=1}^N
f(g^n)(\Re\tau(b\, \a^n (b) \, \a^{2n} (b)))_+ \\
&>&0.
\end{eqnarray*}
\end{proof}

%\begin{remark}\label{bohr} An inspection of the above argument shows that one in fact has
%$\Re \tau(a\, \a^n (a) \, \a^{2n} (a)) > 0$ for almost all
%integers $n$ in the Bohr set $K_U$, where ``almost all'' means
%that a set of integers of density zero is removed.
%\end{remark}

\section{Closing remarks}\label{clos-sec}

We present some remarks concerning Problem \ref{silly}.  By
Theorem \ref{thm:gen-triple-aves}, we have a positive answer to
this question when the invariant algebra $\M^\a$ is trivial.  One
can also extend these arguments to cover the case when the
invariant algebra $\M^\a$ is central by representing $\M$ as a
direct integral over $\M^\a$, see Kadison, Ringrose \cite[Chapter
14]{KadRin97}.

It is clear that if the answer to Problem \ref{sillier} is always
positive, then the same is true for Problem \ref{silly}.  What is
less obvious is that the converse is true; if the answer to
Problem \ref{silly} is always true, then the answer to Problem
\ref{sillier} is always true.  To see this, let $(\M,\tau)$ be a
finite von Neumann algebra with two commuting shifts
$\a_1,\a_2$.  We then form the infinite tensor product $\M^\bbZ :=
\bigotimes_{n \in \bbZ} \M$, which is another finite von Neumann algebra, which contains an embedded copy of $\M$ by
using the $0$ coordinate of $\bbZ$.  Next, let $G$ be the free
abelian group on two generators $e, f$, and let $U$ be the action
of $G$ on $\M^\bbZ$ defined by
$$ U(e) \bigotimes_{n \in \bbZ} a_n := \bigotimes_{n \in \bbZ} \a_1^2 \a_2^{-1} (a_n) $$
and
$$ U(f) \bigotimes_{n \in \bbZ} a_n := \bigotimes_{n \in \bbZ} a_{n-1}$$
for all $a_n \in \M$ with all but finitely many $a_n$ equal to $1$.
If we define a shift $\a'$ to $\M^\bbZ$ by the formula
$$ \a' \bigotimes_{n \in \bbZ} a_n := \bigotimes_{n \in \bbZ} \a_1^{2(n+1)} \a_2^{-n} (a_n)$$
we then observe the identities
$$ \a' U(e) (\a')^{-1} = U(e)$$
and
$$ \a' U(f) (\a')^{-1} = U(fe)$$
(here we use the hypothesis that $\a_1, \a_2$ commute).  Because
of this, we can define a shift $\a$ on the crossed product $\M^\bbZ
\rtimes_U G$ by declaring $\a$ to equal $\a'$ on $\M^\bbZ$, and
$$ \a(e) := e; \quad \a(f) := fe.$$
If $a_1, a_2$ lie in $\M^\bbZ$, we observe that
$$ \a^n( a_1 f^2 ) \a^{2n}( f^{-2} a_2 f ) = (\a')^{n}(a_1) ( (\a')^{2n} U(e)^{-2n} (a_2) ) f.$$
If we assume that $a_1$, $a_2$ in fact lie in $\M$, we can simplify this as
$$ \a_1^{2n} (a_1) \a_2^{2n} (a_2) f.$$
Thus, if we assume Problem \ref{silly} has an affirmative answer
for the system $\M^\bbZ \rtimes_U G$, we see that the averages of
$\a_1^{2n} (a_1) \a_2^{2n} (a_2) f$ (and hence of $\a_1^{2n}
(a_1) \a_2^{2n} (a_2)$) converge for arbitrary $a_1, a_2 \in
\M$; from this one easily deduces (after dividing $n$ into even
and odd classes) that Problem \ref{sillier} has an affirmative
answer for the system $\M$.

In particular, we see that the task of establishing Problem
\ref{silly} in the affirmative for arbitrary von Neumann dynamical
systems is at least as hard as that of achieving convergence for
two commuting shifts in the abelian case, a result first obtained
in \cite{conze-lesigne}.

One can also cover some other (non-ergodic, non-abelian) cases of
Problem \ref{silly} by \emph{ad hoc} methods.  Suppose for
instance that $\M$ is a group von Neumann algebra $LG$, with shift
$\alpha$ given by automorphisms $\alpha_1, \alpha_2: G \to G$ of
the group.  Then one can affirmatively answer Problem \ref{silly}
as follows.  Firstly, by density and linearity we may assume that
$a_1,a_2$ are themselves group elements: $a_1 = g_1 \in  G$, $a_2
= g_2 \in G$.  We then see that the means of $\alpha^n (g_1)
\alpha^{2n} (g_2)$ will converge to zero unless there exists a
group element $g_0$ for which
\begin{equation}\label{goo}
\alpha^n (g_1) \alpha^{2n} (g_2) = g_0
\end{equation}
for all $n$ in a set of positive upper density.  But such sets contain
non-trivial parallelograms $n, n+h, n+k, n+h+k$ for $h,k > 0$.
Applying \eqref{goo} for $n, n+h$ and rearranging, one obtains
$$ \alpha^n ( g_2 \alpha^{2h} (g_2^{-1}) ) = g_1^{-1} \alpha^h (g_1).$$
Similarly, applying \eqref{goo} for $n+k,n+h+k$ one has
$$ \alpha^{n+k} ( g_2 \alpha^{2h} (g_2^{-1}) ) = g_1^{-1} \alpha^h (g_1).$$
Writing $u := g_1^{-1} \alpha^h (g_1)$, one thus has
$$ \alpha^h (g_1) = g_1 u$$
and
$$ \alpha^k (u) = u.$$
If we then write
$$ v := g_1^{-1} \alpha^{hk} (g_1) = u \alpha^h (u) \ldots \alpha^{(k-1) h} (u)$$
we see that
$$ \alpha^{hkn} (g_1) = g_1 v^n$$
for all $n$, and $\alpha (v) = v$.  Thus we have
$$ \alpha^{hkn+j} (g_1) \alpha^{2hkn+2j} (g_2) = \alpha^j( g_1 (\alpha^{2hk} (v))^n \alpha^j (g_2) )$$
for any $n,j$. The means of this in $n$ converge in $L^2(\tau)$ by
the mean ergodic theorem.  Summing over all $0 \leq j < hk$ we
obtain weak convergence, thus answering Problem \ref{silly}
affirmatively in this case.  The same type of argument also lets one
deal with crossed products of abelian systems by groups, in which
the shift acts as an automorphism on the group; we omit the details.

Finally, we remark that the results on asymptotically abelian
systems, while stated for $\bbZ^k$-systems, should in fact be valid
for any commuting action of a general locally compact second
countable (lcsc) abelian group.

\appendix

\section{An application of the van der Corput lemma}\label{duv}

The purpose of this appendix is to establish Theorem \ref{thm:rel-w-m-red}.  Our arguments follow
\cite[Proposition 7.4, Theorem 7.5]{NicStrZsi03} closely (see also \cite[Proposition 4.4]{BeyDuvStr07} for another adaptation of the same argument).  We may normalise $\a_0$ to be the identity.

We induct on $k\geq 2$. When $k=2$ we know from the usual mean
ergodic theorem for von Neumann algebras (see e.g. \cite[Section 9.1]{Kre85}) that
\[\frac{1}{N}\sum_{n=1}^N \a^n(a) \to E_{\M^\a}(a)\quad\quad\hbox{in}\ \|\cdot\|_{L^2(\tau)},\]
and since $\M^\a \subseteq \N$ by the relative weak mixing assumption,
we also have
\[\frac{1}{N}\sum_{n=1}^N \a^n(E_\N(a)) \to E_{\M^\a}(E_\N(a)) =
E_{\M^\a}(a)\quad\quad\hbox{in}\ \|\cdot\|_{L^2(\tau)},\] so
combining these conclusions gives the result.

Now suppose that $k \geq 3$ and that we know the desired
conclusion for any similar family of $\ell < k$ automorphisms.  By
decomposing each $a_i$ as $(a_i - E_\N(a_i)) + E_\N(a_i)$ and
expanding out the expression $\prod_{i=1}^{k-1}\a_i^n(a_i)$, we
find that it suffices to show that for any $i\leq k-1$ we have
\[a_i \perp \N
\quad\quad\Rightarrow\quad\quad\frac{1}{N}\sum_{n=1}^N\prod_{i=1}^{k-1}\a_i^n(a_i)
\to 0\quad\quad\hbox{in}\ \|\cdot\|_{L^2(\tau)};\] let us argue
the case $i=1$, the others following at once by symmetry.

By the Hilbert-space-valued version of the classical van der Corput
estimate (see, for instance, \cite{Fur77} or \cite{Ber87}) this will
follow if we show that
\begin{multline*}
\frac{1}{H}\sum_{h=1}^H\left|\frac{1}{N}\sum_{n=1}^N\Big\langle\prod_{i=1}^{k-1}\a_i^{n+h}(a_i),
\prod_{i=1}^{k-1}\a_i^n(a_i)\Big\rangle_\tau\right|\\
=
\frac{1}{H}\sum_{h=1}^H\left|\frac{1}{N}\sum_{n=1}^N\tau\big(\a_{k-1}^n(\a_{k-1}^h(a_{k-1}^\ast))\cdots
\a_1^n(\a_1^h(a_1^\ast))\cdot \a_1^n(a_1)\cdots \a_{k-1}^n(a_{k-1})\big)\right| \to
0
\end{multline*}
as $N\to\infty$ and then $H\to\infty$.

Let us now set $b_i := \a_i^n(\a_i^h(a_i^\ast))$ and $c_i :=
\a_i^n(\a_i^h(a_i))$ to lighten notation.  Having done so, we now
set ourselves up for applying the asymptotic abelianness property
by observing that
\begin{eqnarray*}
b_{k-1}b_{k-2}b_{k-3}\cdots c_1c_2\cdots &=& (b_{k-2} b_{k-1} b_{k-3}\cdots
c_1c_2\cdots) +
([b_{k-1},b_{k-2}]b_{k-3}\cdots c_1c_2\cdots)\\
&=& (b_{k-2} b_{k-3}b_{k-1}b_{k-4}\cdots c_1c_2\cdots) + (b_{k-2}[b_{k-1},b_{k-3}]b_{k-4}\cdots c_1c_2\cdots )\\
&&\quad\quad\quad\quad + ([b_{k-1},b_{k-2}]b_{k-3}b_{k-4}\cdots
c_1c_2\cdots)\\
&\vdots& \\
&=& b_{k-2}b_{k-3}b_{k-4}\cdots b_1c_1c_2\cdots c_{k-2}(b_{k-1}c_{k-1})\\
&&\quad\quad\quad\quad+\sum_{j=1}^{k-2}x_j[b_{k-1},b_j]y_j +
\sum_{j=1}^{k-2}u_j[b_{k-1},c_j]v_j
\end{eqnarray*}
where each $x_j$, $y_j$, $u_j$ and $v_j$ for $1 \leq j \leq k - 2$
is some product of a subset of the elements $\{b_i,c_i:\ i\leq
k-2\}$.

Importantly, there is some $M>0$ such that $\|x_j\|, \|y_j\|, \|u_j\|, \|v_j\| \leq M$ for
all $j\leq k-2$, and not depending on $n$ or $h$, while on the other
hand for any $j \leq k-2$ we have
\[[b_{k-1},b_j] = [\a_{k-1}^n(\a_{k-1}^h(a_{k-1}^\ast)),\a_j^n(\a_j^h(a_j^\ast))],\]
and hence overall we have
\begin{multline*}
\frac{1}{N}\sum_{n=1}^N\Big\|\sum_{j=1}^{k-2}x_j[b_{k-1},b_j]y_j\Big\|_{L^2(\tau)}
\leq
M^{2}\sum_{j=1}^{k-2}\frac{1}{N}\sum_{n=1}^N\|[b_{k-1},b_j]\|_{L^2(\tau)}\\
=
M^{2}\sum_{j=1}^{k-2}\frac{1}{N}\sum_{n=1}^N\|[\a_{k-1}^h(a_{k-1}^\ast),(\a_{k-1}^{-1}\a_j)^n
(\a_j^h(a_j^\ast))]\|_{L^2(\tau)}
\to 0
\end{multline*}
as $N\to\infty$, by the asymptotic abelianness of $\a_{k-1}^{-1}\a_j$.
The same reasoning applies to the term
$\sum_{j=1}^{k-2}u_j[b_{k-1},c_j]v_j$, and now applies again to show
that in the scalar average of interest to us we may also commute
$b_{k-2}$ from the left-hand-end of our product over to be
immediately on the left of $c_{k-2}$, and then move $b_{k-3}$ to
$c_{k-3}$, and so on.  Overall, this shows that
\begin{eqnarray*}
&&\frac{1}{H}\sum_{h=1}^H\left|\frac{1}{N}\sum_{n=1}^N\tau\big(\a_{k-1}^n(\a_{k-1}^h(a_{k-1}^\ast))\cdots
\a_1^n(\a_1^h(a_1^\ast))\cdot \a_1^n(a_1)\cdots \a_{k-1}^n(a_{k-1})\big)\right|\\
&&\sim
\frac{1}{H}\sum_{h=1}^H\left|\frac{1}{N}\sum_{n=1}^N\tau\big(\a_1^n(\a_1^h(a_1^\ast)a_1)\cdots
\a_{k-1}^n(\a_{k-1}^h(a_{k-1}^\ast)a_{k-1})\big)\right|\\
&&=
\frac{1}{H}\sum_{h=1}^H\left|\frac{1}{N}\sum_{n=1}^N\tau\big(\a_1^h(a_1^\ast)a_1\cdot
(\a_2\a_2^{-1})^n(\a_2^h(a_2^\ast)a_2)\cdots
(\a_{k-1}\a_1^{-1})^n(\a_{k-1}^h(a_{k-1}^\ast)a_{k-1})\big)\right|\\
&&=
\frac{1}{H}\sum_{h=1}^H\left|\tau\Big(\a_1^h(a_1^\ast)a_1\cdot\Big(\frac{1}{N}\sum_{n=1}^N
(\a_2\a_1^{-1})^n(\a_2^h(a_2^\ast)a_2)\cdots
(\a_{k-1}\a_1^{-1})^n(\a_{k-1}^h(a_{k-1}^\ast)a_{k-1})\Big)\Big)\right|
\end{eqnarray*}
as $N\to \infty$ and then $H\to\infty$.  However, now we notice that
the inner average of operators with respect to $N$ here is precisely
of the form hypothesized by the theorem, but involving only the
$k-1$ automorphisms $\a_j\a_1^{-1}$, $j=1,2,\ldots,k-1$, which still
satisfy the necessary hypotheses of relative weak mixing and asymptotic
abelianness. Hence this operator average asymptotically agrees with
\begin{multline*}
\frac{1}{H}\sum_{h=1}^H\left|\tau\Big(\a_1^h(a_1^\ast)a_1\cdot\Big(\frac{1}{N}\sum_{n=1}^N
(\a_2\a_1^{-1})^n(E_\N(\a_2^h(a_2^\ast)a_2))\cdots
(\a_{k-1}\a_1^{-1})^n(E_\N(\a_{k-1}^h(a_{k-1}^\ast)a_{k-1}))\Big)\Big)\right|\\
=
\frac{1}{H}\sum_{h=1}^H\left|\tau\Big(E_\N(\a_1^h(a_1^\ast)a_1)\cdot\Big(\frac{1}{N}\sum_{n=1}^N
(\a_2\a_1^{-1})^n(E_\N(\a_2^h(a_2^\ast)a_2))\cdots
(\a_{k-1}\a_1^{-1})^n(E_\N(\a_{k-1}^h(a_{k-1}^\ast)a_{k-1}))\Big)\Big)\right|,
\end{multline*}
where the second equality holds because the operator average in the
inner brackets now lies in $\N$, and so we apply the usual
identity for conditional expectations $\tau(aE_\N(b)) =
\tau(E_\N(aE_\N(b))) = \tau(E_\N(a)E_\N(b))$.

Writing
\[s_N := \frac{1}{N}\sum_{n=1}^N
(\a_2\a_1^{-1})^n(E_\N(\a_2^h(a_2^\ast)a_2))\cdots
(\a_{k-1}\a_1^{-1})^n(E_\N(\a_{k-1}^h(a_{k-1}^\ast)a_{k-1})),\] we
see that $\|s_N\| \leq C$ for some fixed $C$ and all $N \in \bbN$,
and now combining this bound with the Cauchy-Schwarz inequality we
obtain
\begin{eqnarray*}
\frac{1}{H}\sum_{h=1}^H\left|\tau(E_\N(\a_1^h(a_1^\ast)a_1)\cdot
s_n)\right|
%\leq \frac{1}{H}\sum_{h=1}^H\big|\tau(E_\N(\a_1^h(a_1^\ast)a_1)\cdot
%s_n)\Big|\\
&=& \frac{1}{H}\sum_{h=1}^H\big|\big\langle
s_n^*\hat 1, (E_\N(\a_1^h(a_1^\ast)a_1)\hat 1\big\rangle_{L^2(\tau)}\big|\\
&\leq&
\frac{1}{H}\sum_{h=1}^H C\cdot \|E_\N(\a_1^h(a_1^\ast)a_1\|_{L^2(\tau)}.
\end{eqnarray*}
Finally, it follows that this tends to $0$ as $H\to\infty$ by the
our assumption that $a_1 \perp \N$ and the relative weak mixing
hypothesis.  This completes the proof of
Theorem~\ref{thm:rel-w-m-red}.

\section{A group theory construction}\label{groupthy}

The purpose of this appendix is to explicitly describe a certain
type of group, which we shall term a \emph{square group}, generated
by relations involving quadruples of generators.  In particular, we
will be able to solve the equality problem for such groups.  Our
arguments here are motivated by an observation of Grothendieck that
groups can be identified with the sheaf of their flat connections on
simplicial complexes, and experts will be able to detect the ideas
of sheaf theory lurking beneath the surface of the material here,
although we will not use that theory explicitly.

\begin{definition}[Square groups]  A \emph{square base} $\Box = ( H \cup V, \Box )$ consists of the following data:
\begin{itemize}
\item A set $H \cup V$ of generators, partitioned into a subset
$H$ of \emph{horizontal generators} and a subset $V$ of
\emph{vertical generators}; \item A set $\Box \subset (H \times V
\times H \times V) \cup (V \times H \times V \times H)$ of
quadruples $(e_0,e_1,e_2,e_3)$ of alternating orientation (thus if
$e_0$ is horizontal then $e_1$ must be vertical, and so forth).
\end{itemize}
Furthermore, we require the following two axioms on the set $\Box$:
\begin{itemize}
\item (Cyclic symmetry) If $(e_0,e_1,e_2,e_3) \in \Box$, then
$(e_1,e_2,e_3,e_0) \in \Box$.
\item (Unique continuation) If $e_0,
e_1 \in H \cup V$, then there is at most one quadruple
$(e_0,e_1,e_2,e_3) \in \Box$ with the first two components $e_0$
and $e_1$.
\end{itemize}
If $\Box$ is a square base, we define the \emph{square group}
$G_\Box$ associated to that base to be the group generated by the
generators $H \cup V$, subject to the relations $e_0 e_1 e_2 e_3 =
\id$ for all $(e_0,e_1,e_2,e_3) \in \Box$.  We define the
\emph{alphabet} of the square base (or square group) to be the set
$H \cup V \cup H^{-1} \cup V^{-1}$ consisting of the horizontal
and vertical generators and their formal inverses.
\end{definition}

To describe square groups explicitly, we shall need some notation
of a combinatorial and geometric nature.  Let $\bbN :=
\{0,1,2,\ldots\}$ denote the natural numbers.

\begin{definition}[Monotone paths and regions]  A \emph{monotone path} is a finite path in the discrete
quadrant $\bbN^2$ from $(0,0)$ to some endpoint $(n,m)$ that
consists only of rightward edges $(i,j) \to (i+1,j)$ and upward
edges $(i,j) \to (i,j+1)$ (in particular, the path will have
length $n+m$).  Given a monotone path $\gamma$ from $(0,0)$ to
$(n,m)$, the \emph{shadow} of $\gamma$ is defined to be all the
pairs $(i,j) \in \bbN^2$ such that $(i,j') \in \gamma$ for some
$j' \geq j$.  We say that one monotone path $\gamma'$ \emph{lies
above} another monotone path $\gamma$ with the same endpoint
$(n,m)$ if the shadow of $\gamma'$ contains the shadow of
$\gamma$.  In such cases, we refer to the set-theoretic difference
between the two shadows as a \emph{monotone region} from $(0,0)$
to $(n,m)$, with $\gamma'$ and $\gamma$ referred to as the
\emph{upper boundary} and \emph{lower boundary} of the region
respectively.
\end{definition}

\begin{figure}[tb] \centerline{\includegraphics{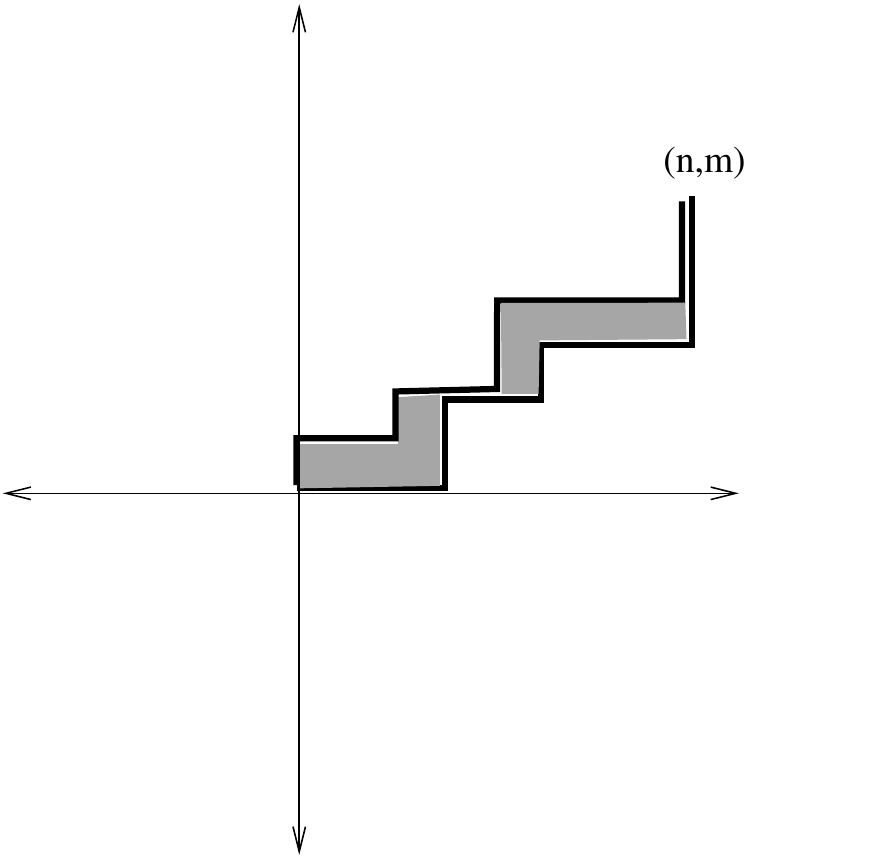}}
\caption{A monotone region, bounded above and below by two monotone
paths.  Note the horizontal and vertical convexity of the monotone
region.} \label{monotone-fig}
\end{figure}

We will also consider a monotone path as a degenerate example of a
monotone region. Monotone regions are horizontally and vertically
convex: if two endpoints of a horizontal or vertical line segment in
$\bbN^2$ lie in a monotone region, then the interior of that segment
does also.

\begin{definition}[Flat connections]\label{flatcon}  Fix a square base $\Box$, and let $\Omega \subset \bbN^2$
be a set.  A \emph{connection} $\Gamma$ on $\Omega$ is an
assignment $\Gamma((i,j) \to (i+1,j)) \in H \cup H^{-1}$ of a
horizontal element of the alphabet to every horizontal edge
$(i,j), (i+1,j) \in \Omega$, and an assignment $\Gamma((i,j) \to
(i,j+1)) \in V \cup V^{-1}$ of a vertical element of the alphabet
to every vertical edge $(i,j) \mapsto (i,j+1) \in \Omega$.  We
adopt the convention that $\Gamma((i+1,j) \to (i,j)) :=
\Gamma((i,j) \to (i+1,j))^{-1}$ and $\Gamma((i,j+1) \to (i,j)) :=
\Gamma((i,j) \to (i,j+1))^{-1}$, where $(e^{-1})^{-1} := e$ for $e
\in H \cup V$ of course.

We say that the connection $\Gamma$ is \emph{flat} if for every
square $(i,j), (i+1,j), (i,j+1), (i+1,j+1)$ in $\Omega$, there
exists an oriented loop $f_0, f_1, f_2, f_3$ of horizontal and
vertical edges around the square (in either orientation) such that
$(\Gamma(f_0), \Gamma(f_1), \Gamma(f_2), \Gamma(f_3)) \in \Box$.

A flat connection on a monotone region from $(0,0)$ to $(n,m)$ is
said to be \emph{maximal} if it cannot be extended to any strictly
larger monotone region with the same endpoints.  It is
\emph{reduced} if there does not exist a triple $(i,j), (i+1,j),
(i+2,j)$ or $(i,j), (i,j+1), (i,j+2)$ in $\Omega$ such that
$\Gamma( (i,j) \to (i+1,j) ) \Gamma( (i+1,j) \to (i+2,j) ) = \id$
or $\Gamma( (i,j+1) \to (i,j) ) \Gamma( (i,j+1) \to (i,j+2) ) =
\id$.
\end{definition}

\begin{figure}[tb]
\centerline{\includegraphics{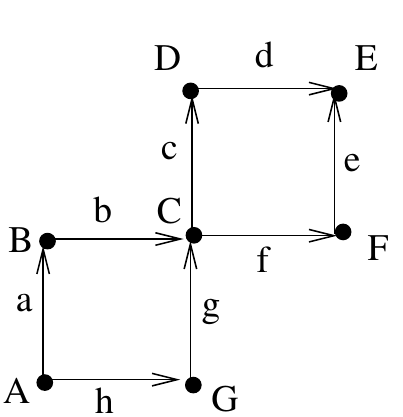}} \caption{A monotone
region $\{A,B,C,D,E,F,G\}$ (with $A=(0,0)$, $B=(0,1)$, etc.) with a
connection $\Gamma$ defined by the group elements $a,b,c,d,e,f,g,h
\in G_\Box$, thus for instance $\Gamma(B \to C) = b$ and $\Gamma(C
\to B) = b^{-1}$.  If for instance $(a,b,g^{-1},h^{-1})$ and
$(f,e,d^{-1},c^{-1})$ are in $\Box$, then this connection is flat.}
\label{connection-fig}
\end{figure}

In the degenerate case when $\Omega$ is just a monotone path,
every connection is automatically flat, as there are no squares.

Let $\Gamma$ be a flat connection on a monotone region $\Omega$.
Then one can \emph{integrate} this connection to produce a map
$\Phi_\Gamma: \Omega \to G_\Box$ by setting $\Phi_\Gamma(0,0) :=
\id$ and $\Phi_\Gamma( v ) = \Phi_\Gamma(u) \Gamma(u \to v)$ for
all horizontal and vertical edges $(u \to v)$ in $\Omega$.  From
the flatness of $\Gamma$ and the ``connected'' nature of $\Omega$
it is easy to see that $\Phi_\Gamma$ exists and is unique.  In
particular, we can define the \emph{definite integral} $|\Gamma|$
of $\Gamma$ to be the group element $|\Gamma| :=
\Phi_\Gamma(n,m)$, where $(n,m)$ is the endpoint of $\Omega$.

\begin{example}  The definite integral of the flat connection in Figure \ref{connection-fig} is equal
to $abcd = abfe = hgcd = hgfe$. \fin
\end{example}

Observe that every group element $g$ in $G_\Box$ can arise as a
definite integral of some flat connection, simply by expressing
$g$ as a word in the alphabet $H \cup V \cup H^{-1} \cup V^{-1}$,
and creating an associated monotone path and connection for that
word.  Later on we shall see that the definite integral will
provide a one-to-one correspondence between group elements and
maximal reduced flat connections (Corollary \ref{maxred}).

We have the following fundamental facts:

\begin{lemma}\label{unimax} Let $\Box$ be a square base, and let $(n,m) \in \bbN^2$.
\begin{itemize}
\item (Unique continuation) If $\Omega$ is a monotone region from
$(0,0)$ to $(n,m)$, and $\gamma$ is a path from $(0,0)$ to $(n,m)$
in $\Omega$, then any flat connection on $\Omega$ is uniquely
determined by its restriction to $\gamma$. In other words, if
$\Gamma, \Gamma'$ are two flat connections on $\Omega$ that agree
on $\gamma$, then they agree on all of $\Omega$. \item
(Maximality)  If $\Omega_0$ is a monotone region from $(0,0)$ to
$(n,m)$, and $\Gamma$ is a flat connection on $\Omega_0$, then
there exists a unique extension of $\Gamma$ to a maximal flat
connection on a monotone region $\Omega$ from $(0,0)$ to $(n,m)$
containing $\Omega_0$.
\end{itemize}
\end{lemma}

\begin{proof} We first establish unique continuation.  This is best explained visually.  The key observation
is that if two flat connections on a square agree on two adjacent
sides of a square, then they must agree on the whole square.  This
is ultimately a consequence of the unique continuation property of
the square base $\Box$, and can be verified by a routine case check.
Thus, if $\Gamma, \Gamma'$ are two connections on $\Omega$ that
agree on $\gamma$, they also agree on any perturbation of $\gamma$
in $\Omega$ formed by taking an adjacent pair of horizontal and
vertical edges in $\gamma$ and ``popping'' them by replacing them by
the other two edges of the square that they form; note that this
retains the property of being a monotone path.  One can check that
after a sufficient number of upward and downward ``popping''
operations one can cover the upper and lower boundaries of $\Gamma$,
and everything in between, and the claim follows.

\begin{example}  We continue working with Figure \ref{connection-fig}.  Suppose two flat connections
$\Gamma, \Gamma'$ on the indicated region agree on the upper
boundary $ABCDE$, with the indicated connection values $a,b,c,d$.
By unique continuation of $\Box$, the only possible values
available for $\Gamma, \Gamma'$ on the remaining two edges $CF$,
$FE$ of the square $CDEF$ are $f$ and $e$.  Thus we may ``pop''
the upper square and obtain that $\Gamma$, $\Gamma'$ also agree on
the monotone path $ABCFE$.  After popping the lower square also we
obtain that $\Gamma, \Gamma'$ agree on the entire monotone region.
\end{example}

To prove the second claim, we simply observe that if $\Gamma$ can
be extended to two monotone regions $\Omega, \Omega'$ containing
$\Omega_0$, then by unique continuation they agree on the
intersection $\Omega \cap \Omega'$ (which is also a monotone
region), and can thus be glued to form a flat connection on the
union $\Omega \cup \Omega'$ (which is also a monotone
region\footnote{One way to see this is to rotate the plane by $45$
degrees, so that monotone paths become graphs of discrete
Lipschitz functions with Lipschitz constant $1$, and monotone
regions become the regions between two such functions.}).  Since
there are only finitely many monotone regions from $(0,0)$ to
$(n,m)$, the claim then follows from the greedy algorithm.
\end{proof}

Now we need a fundamental definition.

\begin{definition}[Concatenation]\label{concat}  Let $\Gamma$ be a maximal reduced flat connection
on some monotone region $\Omega$ from $(0,0)$ to $(n,m)$, and let
$x \in H \cup V \cup H^{-1} \cup V^{-1}$ be a symbol in the
alphabet.  We define the \emph{concatenation} $\Gamma \cdot x$ of
$\Gamma$ with $x$ to be the maximal flat connection $\Gamma' =
\Gamma \cdot x$ on a monotone region $\Omega'$ from $(0,0)$ to
$(n',m')$ generated by the following rule.
\begin{itemize}
\item (Collapse) If $x$ is horizontal (i.e. $x \in H \cup
H^{-1}$), $(n-1,m)$ lies in $\Omega$, and $\Gamma((n-1,m) \to
(n,m)) = x^{-1}$, then one sets $(n',m') := (n-1,m)$, sets
$\Omega'$ to be the restriction of $\Omega$ to the region $\{
(i,j) \in \bbN^2: i \leq n-1 \}$ (i.e. one deletes the rightmost
column of $\Omega$, and sets $\Gamma'$ to be the restriction of
$\Gamma$ to $\Omega'$. \item (Extension) If $x$ is horizontal, and
either $(n-1,m)$ lies outside of $\Omega$ or $\Gamma((n-1,m) \to
(n,m)) \neq x^{-1}$, then one sets $(n',m') := (n+1,m)$, and
extends $\Gamma$ to $\Omega \cup \{(n+1,m)\}$ by setting
$\Gamma((n,m) \to (n+1,m)) := x$; note that this is still flat
because it does not create any squares.  One then extends $\Gamma$
further by the second part of Lemma \ref{unimax} to create the
maximal flat connection $\Gamma'$ on $\Omega'$ that extends
$\Gamma$. \item If $x$ is vertical instead of horizontal, one
follows the analogue of the above rules but with the roles of $n$
and $m$ reversed.
\end{itemize}
\end{definition}

\begin{example} Imagine one concatenated a horizontal edge $x$ to the flat connection in Figure
\ref{connection-fig}, which we shall assume to be maximal reduced.
If $x$ is not equal to $d^{-1}$, then the concatenated connection
would thus extend one unit to the right of $E$ to the endpoint
$(3,2)$, and may possibly extend also to the square to the right
of $EF$ if there is an appropriate tuple in $\Box$ to achieve this
extension.  If instead $x$ was equal to $d^{-1}$, then the
connection would collapse to the region $\{A,B,C,D,G\}$, so that
the endpoint is now $D=(1,2)$.\fin
\end{example}

The importance of this definition lies in the fact that it gives a representation of $G_\Box$:

\begin{lemma}  Let $\Box$ be a square base, and let $\Gamma$ be a maximal reduced flat connection.
\begin{itemize}
\item (Preservation of reducibility) For any $x \in H \cup V \cup
H^{-1} \cup V^{-1}$, $\Gamma \cdot x$ is reduced.
\item
(Invertibility) For any $x \in H \cup V \cup H^{-1} \cup V^{-1}$,
one has $(\Gamma \cdot x) \cdot x^{-1} = \Gamma$.
\item (Square
relations) For any $(e_0,e_1,e_2,e_3) \in \Box$, one has
$(((\Gamma \cdot e_0) \cdot e_1) \cdot e_2) \cdot e_3 = \Gamma$.
\end{itemize}
In particular, the group $G_\Box$ acts on the space ${\mathcal O}$
of maximal reduced flat connections in a unique manner, sending
$\Gamma$ to $\Gamma \cdot g$ for any $\Gamma \in {\mathcal O}$ and
$g \in G_\Box$.
\end{lemma}

\begin{proof}  We begin with the preservation of reducibility claim.  If $\Gamma \cdot x$ is formed
by collapsing $\Gamma$, the claim is clear, so suppose instead
that $\Gamma \cdot x$ is formed by extension.  By symmetry we may
assume that $x$ is horizontal.  Let $(n,m)$ denote the endpoint of
$\Gamma$, and let $\Omega'$ be the domain of $\Gamma \cdot x$
(which then has endpoint $(n+1,m)$).

Assume for contradiction that $\Gamma \cdot x$ is not reduced.
Since $\Gamma$ was reduced, there are only two possibilities:
either one has a vertical degeneracy
\begin{equation}\label{verden}
 \Gamma( (n+1,j) \to (n+1,j+1) ) \Gamma( (n+1,j+1) \to (n+1,j+2) ) = \id
\end{equation}
for some $(n+1,j), (n+1,j+1), (n+1,j+2) \in \Omega'$, or else one has a horizontal degeneracy
\begin{equation}\label{horden}
\Gamma( (n-1,j) \to (n,j) ) \Gamma( (n,j) \to (n+1,j) ) = \id
\end{equation}
for some $(n-1,j), (n,j), (n+1,j) \in \Omega'$.

Suppose first that one has a vertical degeneracy \eqref{verden}.
Consider the restrictions $\Gamma_1, \Gamma_2$ of the connection
$\Gamma$ on the adjacent squares $( (n,j), (n,j+1), (n+1,j),
(n+1,j+1) )$ and $( (n,j+1), (n,j+2), (n+1,j+1), (n+1,j+2) )$.  By
construction, $\Gamma_1, \Gamma_2$ agree on their common edge $(
(n,j+1) \to (n+1,j+1) )$, and $\Gamma_1( (n+1,j+1) \to (n+1,j) )$
is equal to $\Gamma_2( (n+1,j+1) \to (n+1,j+2) )$. By the unique
continuation property of $\Box$, this implies that $\Gamma_1$ and
$\Gamma_2$ are reflections of each other, and in particular that
$\Gamma_1( (n,j+1) \to (n,j) )$ is equal to $\Gamma_2( (n,j+1) \to
(n,j+2) )$.  But this implies that $\Gamma$ is not reduced, a
contradiction.

Now suppose instead that one has a horizontal degeneracy
\eqref{horden}.  From Definition \ref{concat} we know that $j$
cannot equal $m$, otherwise we would have collapsed rather than
extended $\Gamma$.  Let $0 \leq j < m$ be the largest $j$ for
which \eqref{horden} holds.  By repeating the argument in the
previous paragraph, we see that the restrictions of $\Gamma$ to
the adjacent squares $\{ (n-1,j), (n,j), (n-1,j+1), (n,j+1) \}$
and $\{ (n,j), (n+1,j), (n,j+1), (n+1,j+1) \}$ are reflections of
each other, which implies that \eqref{horden} also holds for
$j+1$, contradicting the maximality of $j$.  This establishes the
preservation of reducibility.

Now we establish the invertibility.  Again, by symmetry we may
assume that $x$ is horizontal.

If $\Gamma \cdot x$ is a (horizontal) extension of $\Gamma$, then
it is easy to see from Definition \ref{concat} that $(\Gamma \cdot
x) \cdot x^{-1}$ will be the (horizontal) collapse of $\Gamma
\cdot x$, which is $\Gamma$.    Conversely, if $\Gamma \cdot x$ is
the (horizontal) collapse of $\Gamma$, then $(\Gamma \cdot x)
\cdot x^{-1}$ will be the (horizontal) extension (because $\Gamma$
was reduced), which will equal $\Gamma$ again (by uniqueness of
maximal extension).

Finally, we establish the square relations.  From cyclic symmetry
and invertibility we may assume that $e_0, e_2$ are horizontal and
$e_1, e_3$ are vertical.  From invertibility again, it suffices to
show that
$$ (\Gamma \cdot e_0) \cdot e_1 = (\Gamma \cdot e_3^{-1}) \cdot e_2^{-1}$$
for any maximal reduced flat connection $\Gamma$.  We use $(n,m)$ to denote the endpoint of $\Gamma$.

We divide into four cases.  Suppose first that $\Gamma \cdot e_0$
is an extension of $\Gamma$, and that $(\Gamma \cdot e_0) \cdot
e_1$ is an extension of $\Gamma \cdot e_0$.  Then we claim that
$\Gamma \cdot e_3^{-1}$ is an extension of $\Gamma$.  For if this
were not the case, then $\Gamma( (n,m-1) \to (n,m) )$ must equal
$e_3$, but then as $(\Gamma \cdot e_0)((n,m) \to (n+1,m))$ equals
$e_0$ by construction, the domain of $\Gamma \cdot e_0$ must
include the square $(n,m-1), (n,m), (n+1,m-1), (n+1,m)$ with
$(\Gamma\ \cdot e_0)((n+1,m-1) \to (n+1,m)) = e_1^{-1}$,
causing $(\Gamma \cdot e_0) \cdot e_1$ to be a collapse rather than
an extension, a contradiction.  Thus $\Gamma \cdot e_3^{-1}$
extends $\Gamma$.  A similar argument shows that $(\Gamma \cdot
e_3^{-1}) \cdot e_2^{-1}$ extends $\Gamma \cdot e_3^{-1}$
(otherwise $\Gamma( (n-1,m) \to (n,m) )$ would equal $e_0^{-1}$,
causing $\Gamma \cdot e_0$ to be a collapse rather than an
extension).  It is then easy to verify that $(\Gamma \cdot
e_3^{-1}) \cdot e_2^{-1}$ and $(\Gamma \cdot e_0) \cdot e_1$ are
the same (since they glue together to form a flat connection on
$\Gamma$ and on the square $(n,m), (n+1,m), (n,m+1), (n+1,m+1)$).

Now suppose that $\Gamma \cdot e_0$ is an extension of $\Gamma$,
but that $(\Gamma \cdot e_0) \cdot e_1$ is a collapse of $\Gamma
\cdot e_0$.  Arguing as before, we conclude that $\Gamma( (n,m-1)
\to (n,m) )$ equals $e_3$, and so $\Gamma \cdot e_3^{-1}$ is a
collapse of $\Gamma$; similarly, $(\Gamma \cdot e_3^{-1}) \cdot
e_2^{-1}$ cannot be a collapse of $\Gamma \cdot e_3^{-1}$ (this
would force $\Gamma \cdot e_0$ to be a collapse also) and so is an
extension.  It is again easy to verify that $(\Gamma \cdot
e_3^{-1}) \cdot e_2^{-1}$ and $(\Gamma \cdot e_0) \cdot e_1$ are
the same.

The remaining two cases (when $\Gamma \cdot e_0$ is a collapse of
$\Gamma$, and $(\Gamma \cdot e_0) \cdot e_1$ is either an
extension or collapse of $\Gamma \cdot e_0$) are similar to the
preceding two, and are left to the reader.
\end{proof}

This gives us a satisfactory explicit description of a square group:

\begin{corollary}\label{maxred}  Let $\Box$ be a square group.  Then the definite integral map
$\Gamma \mapsto |\Gamma|$ is a bijection from ${\mathcal O}$ to
$G_\Box$; thus every group element has a unique representation as
the definite integral of a maximal reduced flat connection.
\end{corollary}

\begin{proof}  The surjectivity of this map was already established in the discussion after
Definition \ref{flatcon}, so it suffices to establish the injectivity.  We will establish this via the identity
$$ \Gamma = \emptyset \cdot |\Gamma| $$
for all $\Gamma \in {\mathcal O}$, where $\emptyset$ is the trivial
flat connection over the monotone region $\{(0,0)\}$ from $(0,0)$ to
$(0,0)$.  This identity shows that $\Gamma$ can be reconstructed
from $|\Gamma|$, demonstrating injectivity.

Let $\Omega$ be the domain of $\Gamma$, which by definition is a
monotone region from $(0,0)$ to some point $(n,m)$.  Let $\gamma$
be some monotone path in $\Omega$ from $(0,0)$ to $(n,m)$ (e.g.
one could take $\gamma$ to be the upper or lower boundary of
$\Omega$).  We label the vertices of $\gamma$ in order as $(0,0) =
(i_0,j_0), (i_1,j_1),\ldots,(i_{n+m},j_{n+m}) = (n,m)$.  From
definition of $|\Gamma|$, we see that
$$
|\Gamma| = \Gamma( (i_0,j_0) \to (i_1,j_1) ) \Gamma( (i_1,j_1) \to
(i_2,j_2) ) \ldots \Gamma( (i_{n+m-1},j_{n+m-1}) \to (i_{n+m},
j_{n+m}) ).$$ For each $0 \leq k \leq n+m$, let $\Omega_k$ be the
portion of $\Omega$ which is in the region $\{ (i,j): i \leq i_k,
j \leq j_k \}$, thus $\Omega_k$ is a monotone region from $(0,0)$
to $(i_k,j_k)$ which is increasing in $k$.  Let $\Gamma_k$ be the
restriction of $\Gamma$ to $\Omega_k$.  As $\Gamma$ was maximal
and reduced, each of the $\Gamma_k$ is also.  Since $\Gamma_{n+m}
= \Gamma$, it will suffice to establish that
$$
\Gamma_k = \emptyset \cdot \Gamma( (i_0,j_0) \to (i_1,j_1) )
\Gamma( (i_1,j_1) \to (i_2,j_2) ) \ldots \Gamma( (i_{k-1},j_{k-1})
\to (i_k, j_k) )$$ for all $0 \leq k \leq n+m$.  But this is
easily established by induction (the reduced nature of the
$\Gamma_k$ is necessary to avoid the collapse case in Definition
\ref{concat}).
\end{proof}

As a consequence of this corollary, we can distinguish any two
elements in $G_\Box$ from each other as long as we can express
them as the definite integrals of distinct maximal reduced flat
connections.

\subsection{Applications}

We now specialise the above abstract group-theoretic machinery to
the application at hand.  We begin with a proposition which will
be used to show non-convergence of quadruple recurrence (Theorem
\ref{k4}).

\begin{proposition}[Independence of AP4 relations]\label{apthy}  Let $A \subset \bbZ$ be a (possibly infinite)
set of integers.  Then there exist a group $G$ with elements
$e_0,e_1,e_2,e_3$, together with an automorphism $T: G \to G$, such
that for $r \in \bbN$, the relation
\begin{equation}\label{slop}
e_0 (T^{r} e_1) (T^{2r} e_2) (T^{3r} e_3) = \id
\end{equation}
holds if and only if $r \in A$.  Furthermore, no power $T^k$ of
$T$ with $k \neq 0$ has any fixed points other than the identity
element $id$.
\end{proposition}

\begin{remark}\label{3slop}
Informally, this proposition asserts that the algebraic relations
\eqref{slop} for various $r \in \bbZ$ are independent of each
other. In contrast, with progressions of length three (i.e. in the case $k=3$) the
analogous relations are highly degenerate.  Indeed, suppose that
\begin{equation}\label{err}
 e_0 (T^r e_1) (T^{2r} e_2) = \id
 \end{equation}
for all $r \in A$.  Then if $r, r+h$ lie in $A$, we have
$$ e_0 (T^r e_1) (T^{2r} e_2) = e_0 (T^r T^h e_1) (T^{2r} T^{2h} e_2)$$
which we can rearrange as
$$ (T^h e_1^{-1}) e_1 = T^{r} ( (T^{2h} e_2) e_2^{-1} ).$$
If $r,r+h,r',r'+h$ lie in $A$, we thus have
$$ T^{r} ( (T^{2h} e_2) e_2^{-1} ) = T^{r'} ( (T^{2h} e_2) e_2^{-1} ).$$
Assuming that $T^{r'-r}$ has no fixed points, we conclude that
$(T^{2h} e_2) e_2^{-1}$ is the identity; assuming that $T^{2h}$ has
no fixed points either, we conclude that $e_2$ is the identity.
Similar arguments can be used to show that $e_0$ and then $e_1$ are
also the identity.  Thus the relations \eqref{err} and the
no-fixed-points hypothesis lead to a total collapse of the group
generated by $e_0,e_1,e_2$ as soon as $A$ contains even a single
non-trivial parallelogram $r, r+h, r', r'+h$.  (A variant of this
argument also shows that if \eqref{err} is obeyed for $r$ and $r+h$,
then it is also obeyed for $r+2h$ even without the fixed point
hypothesis.) This algebraic distinction between triple recurrence
and quadruple recurrence can be viewed as the primary reason why
recurrence and convergence results continue to hold for triple
products, but not for quadruple products even under the assumption
of ergodicity (which is reflected here in the no-fixed-points
assumption).\fin
\end{remark}

\begin{proof}  We let $G$ be the group generated by the generators $e_{i,n}$ for $i=0,1,2,3$ and $n \in \bbZ$,
subject to the relations
$$
e_{0,n} e_{1,n+r} e_{2,n+2r} e_{3,n+3r} = \id$$ for all $n \in \bbZ$
and $r \in A$.  As the set of such relations is invariant under the
shift $e_{i,n} \mapsto e_{i,n+1}$, we see that we can define an
automorphism $T: G \to G$ by setting $T e_{i,n} := e_{i,n+1}$.  If
we then set $e_i := e_{i,0}$, it is clear that \eqref{slop} holds
for all $r \in A$.

To see that \eqref{slop} fails for $r \not \in A$, we observe that
$G$ can be viewed as a square group, with the horizontal generators
$\{ e_{i,n}: i = 0,2; n \in \bbZ \}$ and vertical generators $\{
e_{i,n}: i = 1,3; n \in \bbZ\}$ and square relations $\Box$
consisting of $(e_{0,n}, e_{1,n+r}, e_{2,n+2r}, e_{3,n+3r})$ and its
cyclic permutations for all $n \in \bbZ$ and $r \in A$; note that
the crucial unique continuation property follows from the basic
observation that an arithmetic progression is determined by any two
of its elements (``two points determine a line'').  If $n \in \bbZ$
and $r \not \in A$, one sees that the connection on the path of
length four from $(0,0)$ to $(2,2)$ associated to the word $e_{0,n}
e_{1,n+r} e_{2,n+2r} e_{3,n+3r}$ is already a maximal reduced flat
connection (as none of the three squares that share two edges with
the path can be completed to a square from $\Box$) and so by
Corollary \ref{maxred}, its definite integral $e_{0,n} e_{1,n+r}
e_{2,n+2r} e_{3,n+3r}$ is not equal to the identity, as required.

Finally, to show that $T^k$ has no non-trivial fixed points, one
simply observes that $T^k$ will shift any non-trivial maximal
reduced flat connection to a different maximal reduced flat
connection, and then invokes Corollary \ref{maxred} again.
\end{proof}

Next, we establish a variant that is useful for showing negative
averages for quintuple recurrence (Theorem \ref{negavk}).

\begin{proposition}[Independence of AP5 relations]\label{apthy2}
%Let $A \subset \bbZ$ be a (possibly infinite)
%set of integers.  Then
There exists a group $G$ with distinct elements
$e_0,e_1,e_2,e_3,e_4$, together with an automorphism $T: G \to G$,
such that the relation
\begin{equation}\label{slop2}
e_0 (T^{r} e_1) (T^{2r} e_2) (T^{3r} e_3) (T^{4r} e_4) = \id
\end{equation}
holds for all $r \in \bbZ$.  Furthermore, no power $T^k$ of $T$
with $k \neq 0$ has any fixed points other than the identity
element $id$.  Finally, if $r \in \bbZ$ is nonzero, and
$$g_0,g_1,g_2,g_3,g_4 \in \{ id, e_0, e_1, e_2, e_3, e_4, e_0^{-1}, e_1^{-1}, e_2^{-1}, e_3^{-1}, e_4^{-1} \}$$
are such that
\begin{equation}\label{slop3}
g_0 (T^{r} g_1) (T^{2r} g_2) (T^{3r} g_3) (T^{4r} g_4) = \id,
\end{equation}
then $g_0,g_1,g_2,g_3,g_4$ are either equal to the identity, or
are a permutation of $\{e_0,e_1,e_2,e_3,e_4\}$ or of $\{e_0^{-1},
e_1^{-1}, e_2^{-1}, e_3^{-1}, e_4^{-1} \}$.
\end{proposition}

\begin{proof}  For each $i=0,1,2,3,4$, we define $G^{(i)}$ to be the group generated by the
generators $e_{j,n}^{(i)}$ for $j \in \{0,1,2,3,4\} \backslash
\{i\}$ and $n \in \bbZ$ subject to the relations
\begin{equation}\label{songo}
 e_{0,n}^{(i)} e_{1,n+r}^{(i)} e_{2,n+2r}^{(i)} e_{3,n+3r}^{(i)} e_{4,n+4r}^{(i)} = \id
\end{equation}
for all $n,r \in \bbZ$, with the convention that $e^{(i)}_{i,n} =
\id$ for all $n$.  This group has an automorphism $T^{(i)}: G^{(i)}
\to G^{(i)}$ that maps $e^{(i)}_{j,n}$ to $e^{(i)}_{j,n+1}$ for all
$n$.

We now set $G$ to be the product group $G := G^{(0)} \times G^{(1)} \times \ldots \times G^{(4)}$, and set
$$ e_j := ( e_{j,0}^{(0)}, e_{j,0}^{(1)}, \ldots, e_{j,0}^{(4)} )$$
for $j=0,1,2,3,4$.  We also set
$$ T( g^{(0)},g^{(1)}, \ldots, g^{(4)} ) := ( T^{(0)} g^{(0)}, T^{(1)} g^{(1)}, \ldots, T^{(4)} g^{(4)} ),$$
thus $T$ is an automorphism on $G$.  By construction it is clear
that \eqref{slop2} holds.  Also, by the arguments in Proposition
\ref{apthy}, no non-zero power of $T^{(i)}$ has any non-trivial
fixed points, and so the same is also true of $T$.

Now we establish the final claim of the proposition.  Suppose
$g_0,\ldots,g_4$ obey the stated properties.  Let $i = 0,1,2,3,4$,
and let $g_j^{(i)}$ be the $G^{(i)}$ component of $g_j$ for
$j=0,1,2,3,4$, thus
\begin{equation}\label{slash}
g_0^{(i)} ((T^{(i)})^{r} g_1^{(i)}) ((T^{(i)})^{2r} g_2^{(i)})
((T^{(i)})^{3r} g_3^{(i)}) ((T^{(i)})^{4r} g_4^{(i)}) = \id.
\end{equation}
From construction of $G^{(i)}$, we see that for any distinct $j,k
\in \{0,1,2,3,4\} \backslash \{i\}$, there is a homomorphism
$\phi^{(i)}_{j,k}: G^{(i)} \to \bbZ$ to the additive group $\bbZ$
that maps $e_{j,n}^{(i)}$ to $+1$, $e_{k,n}^{(i)}$ to $-1$, and
all other $e_{l,n}^{(i)}$ to zero for $n \in \bbZ$ and $l \in
\{0,1,2,3,4\} \backslash \{i,j,k\}$ (note that these requirements
are compatible with the defining relations \eqref{songo}).  This
homomorphism is $T^{(i)}$ invariant.  Applying this homomorphism
to \eqref{slash}, we obtain
$$ \sum_{l=0}^4 \phi^{(i)}_{j,k}( g_l^{(i)} ) = 0.$$

In other words, the number of times $g_l$ for $l=0,1,2,3,4$ equals
$e_j$, minus the number of times it equals $e_j^{-1}$, is equal to
the number of times $g_l$ equals $e_k$, minus the number of times
it equals $e_k^{-1}$.  Letting $j,k,i$ vary, we thus see that this
number is independent of $j$.  It is easy to see that this number
cannot exceed $1$ in magnitude, and if it is equal to $+1$ or
$-1$, then $g_0,g_1,g_2,g_3,g_4$ is a permutation of
$\{e_0,e_1,e_2,e_3,e_4\}$ or of $\{e_0^{-1}, e_1^{-1}, e_2^{-1},
e_3^{-1}, e_4^{-1} \}$ respectively.  (Note that this argument also ensures that $e_0,e_1,e_2,e_3,e_4$ are distinct.)  The remaining possibility to
eliminate is when this number is zero, thus each $e_i$ occurs in
$g_0,g_1,g_2,g_3,g_4$ as often as $e_i^{-1}$.  Suppose for
instance that $g_0,g_1,g_2,g_3,g_4$ contains one occurrence each
of $e_0, e_0^{-1}, e_1, e_1^{-1}$.  Applying \eqref{slash} with
$i=4$ (say), and then applying the homomorphism that maps
$e_{0,n}^{(4)}$ to zero, $e_{1,n}^{(4)}$ to $n$, $e_{2,n}^{(4)}$
to $-2n$, and $e_{3,n}^{(4)}$ to $n$ (here we use the identity
$(n+r)-2(n+2r)+(n+3r)=0$ to ensure consistency with \eqref{songo})
we obtain a contradiction.  Similarly if $g_0,g_1,g_2,g_3,g_4$
contains any other combination of one or two distinct pairs
$e_j,e_j^{-1}$.  The remaining case to eliminate is if
$g_0,g_1,g_2,g_3,g_4$ contains $e_j$ and $e_j^{-1}$ twice each for
some $j$, say $j=0$.  Applying \eqref{slash} with $i=4$ again, we
can use Corollary \ref{maxred} to contradict \eqref{slash} (as the
right-hand side is a definite integral of a maximal flat
connection on a horizontal path of length four).  Similarly for
other values of $j$, and the claim follows.
\end{proof}

\bibliographystyle{abbrv}
\bibliography{bibfile}

\end{document}